\documentclass[11pt,reqno]{amsart}
\usepackage[T1]{fontenc}
\usepackage{graphicx, a4wide, amssymb, amsmath}
\usepackage{cite}
\usepackage{color}
\definecolor{MyLinkColor}{rgb}{0,0,0.4} 

\newcommand{\R}{\mathbb{R}}
\newcommand{\RRM}{\mathbb{R}}
\newcommand{\N}{\mathbb{N}}
\newcommand{\bA}{\mathbb{A}}
\newcommand{\clu}{{\mathcal{U}}}

\newcommand{\cla}{\mathcal{A}}
\newcommand{\clp}{{\mathcal{P}}}

\newcommand{\kL}{{\mathcal{L}}}
\newcommand{\kH}{{\mathcal{H}}}
\newcommand{\vdiv}{\mathop{\rm div}}
\newcommand{\be}{\begin{equation}}
\newcommand{\ee}{\end{equation}}
\newcommand{\eps}{\varepsilon}
\newcommand{\e}{\varepsilon}
\newcommand{\PV}{{\rm PV}}
\newcommand{\p}{\partial}
\newcommand{\supp}{{\rm supp\,}}
\newcommand{\re}{{\rm Re\,}}
\newcommand{\oo}{{\overline\omega}}
\newcommand{\wt}{\widetilde}
\newcommand{\fks}{f_1,\ldots,f_k}
\newcommand{\fs}{f,\ldots,f}
\newcommand{\fgs}{f+g,\ldots,f+g}

\newtheorem{lemma}{Lemma}[section]
\newtheorem{thm}[lemma]{Theorem}
\newtheorem{prop}[lemma]{Proposition}
\newtheorem{cor}[lemma]{Corollary}
\theoremstyle{remark}
\newtheorem{rem}[lemma]{Remark}

\numberwithin{equation}{section}
\usepackage[colorlinks=true,linkcolor=MyLinkColor,citecolor=MyLinkColor]{hyperref} 
\begin{document}

\title[Two-phase Stokes flow by capillarity in full 2D space]{Two-phase Stokes flow by capillarity in full 2D space: an approach via hydrodynamic potentials}

\author{Bogdan--Vasile Matioc}
\address{Fakult\"at f\"ur Mathematik, Universit\"at Regensburg,   93040 Regensburg, Deutschland.}
\email{bogdan.matioc@ur.de}
  
\author{Georg Prokert}
\address{ Faculty of Mathematics and Computer Science, Technical
University Eindhoven, The Netherlands.}
\email{g.prokert@tue.nl}


\begin{abstract}
We study the two-phase Stokes flow   driven by surface tension with two fluids of equal viscosity, separated by an asymptotically flat interface with graph geometry. 
 The flow is assumed to be two-dimensional with the fluids filling the entire space.
 We prove well-posedness and parabolic smoothing in Sobolev spaces up to critical regularity. The main technical tools are an analysis of nonlinear singular integral operators arising from the hydrodynamic single-layer potential and abstract results on nonlinear parabolic evolution equations. 
\end{abstract}

\maketitle

\section{Introduction}
One of the standard methods in the analysis of moving boundary problems is the reformulation of these problems as evolution equations in function spaces to which methods of Functional Analysis can be applied, depending on the character of the problem under investigation. The difficulty of this typically consists in the fact that the resulting evolution equations are nonlocal and strongly nonlinear. 
For moving boundary problems with domains of general shape this approach typically involves the transformation of the moving domain to a fixed reference domain by an unknown, time dependent diffeomorphism, and the (explicit or implicit) use of solution operators for boundary value problems with variable coefficients on this reference domain. This approach often implies restrictions to results of perturbation type, i.e. either short-time solutions, or solutions for (in some sense) small data. 
However, this can be avoided in special situations where
\begin{itemize}
    \item the geometry is simpler (e.g. full space, with the moving boundary being a graph), and
    \item the underlying PDE is elliptic and has constant coefficients.
\end{itemize}
In such situations, one can use the classical methods of potential theory to solve the PDEs directly, i.e. without transformations of the domain, and reformulate the moving boundary problem as an evolution equation that involves nonlinear, singular integral operators. 

 This strategy  proved to be  successful for various versions of  the Muskat (or two-phase Hele-Shaw) problem, see e.g. the survey articles \cite{G17, GL20}.  
While the constant coefficient elliptic operator underlying the Muskat problem is simply the Laplacian, the related moving boundary problems of 
quasistationary Stokes flow are based on the Stokes operator (which is also elliptic in a sense that can be made precise).
Concretely,  in this paper we are interested in the following   moving boundary problem of Stokes flow  driven by the capillarity  of  the  
moving interface $t\mapsto\Gamma(t)$ between two fluid phases $\Omega^\pm(t)$ in $\mathbb{R}^2$: 
\be\label{probint}
\left.
\begin{array}{rclll}
\mu\Delta v^\pm-\nabla q^\pm&=&0&\mbox{in $\Omega^\pm(t)$,}\\
\vdiv v^\pm&=&0&\mbox{in $\Omega^\pm(t)$,}\\
v^+&=&v^-&\mbox{on $\Gamma(t)$,}\\{}
[T(v,q)]\tilde\nu&=&-\sigma\tilde\kappa\tilde\nu&\mbox{on $\Gamma(t)$,}\\
(v^\pm,q^\pm)&\to&0&\mbox{for $|x|\to\infty$,}\\
V_n&=&v\cdot\tilde \nu&\mbox{on $\Gamma(t)$.}
\end{array}\right\}
\ee
 Here, $v^\pm:\Omega^\pm(t)\longrightarrow\mathbb{R}^2$ is a vector field representing the velocity of the liquid located in~$\Omega^\pm(t)$  and $q^\pm:\Omega^\pm(t)\longrightarrow\mathbb{R}$ its pressure. 
Furthermore, $\tilde\nu$ is the unit exterior normal to~$\Omega^-(t)$ and $\tilde\kappa$ denotes the curvature of the interface.  
 Moreover, $[T(v,q)]$ denotes the jump of the stress tensor across $\Gamma(t)$, see \eqref{defjump}, \eqref{defT} below. 
 The positive constants $\mu$ and $\sigma$  denote the viscosity of the liquids and the surface tension coefficient of the interface, respectively. 
 We assume that $$\Gamma(t)=\partial\Omega^\pm(t), \qquad\Omega^+(t)\cup\Omega^-(t)\cup \Gamma(t)=\mathbb{R}^2,$$ and 
 that $\Gamma(t)$  is a graph over a suitable straight line. 
 Equation \eqref{probint}$_6$  determines the motion of the interface by prescribing  its normal velocity as coinciding with the normal component of the velocity at $\Gamma(t)$, i.e. the interface is transported by the liquid flow. 
 The interface $\Gamma(t)$ is assumed to be known at time $t=0$.
 
For the Stokes operator, is is possible to set up a treatment of boundary value problems based on so-called hydrodynamic potentials \cite{Lad63} in strict analogy to the potentials for the Laplacian. 
It is this analogy that enables  us to study  the moving boundary problem of two-phase Stokes flow driven by capillarity (at least in 2D and with 
equal viscosity in both phases) along the same lines as for the Muskat problem. 
This has first been exploited in \cite{BaDu98} to obtain an existence result for all positive times, with initial data  that are small in a space of Fourier transforms of bounded measures. 
 To the best of our knowledge, this is the only result available on two-phase Stokes flow in the unbounded geometry considered here.

It is the aim of the present paper to analyze  Problem \eqref{probint} in Sobolev spaces (up to critical regularity) in an $L_2$-based setting. 
The use of these spaces also implies that the interface is asymptotically flat. 

We will establish
\begin{itemize}
    \item existence and uniqueness of maximal solutions with initial data that are arbitrary within our phase space; 
    \item a corresponding semiflow property;
    \item parabolic smoothing up to ${\rm C}^\infty$ of solutions in time and space (away from the initial time);
    \item a criterion for global existence of solutions, or equivalently, a necessary condition for blow-up. 
\end{itemize}
Essentially, these results are obtained by applying the theory of maximal regularity for nonlinear parabolic equations in weighted H\"older spaces  of vector-valued functions presented in~\cite{L95}. 
Since the Stokes flow \eqref{probint}  is driven by capillarity, it turns out that the problem is parabolic ``everywhere'', i.e. the parabolicity condition is just positivity of the surface energy. 
 We emphasize that while for the discussion of the boundary value problem \eqref{probint}$_{1\--5}$ at a fixed time $t$ we  need to assume $H^3$-smoothness  of the interface (cf. Section~\ref{fixtime}),
the corresponding nonlinear evolution equation~\eqref{NNEP} is shown to be well-posed in  all subcritical spaces~$H^s$ with~${s>3/2}$. 
Hence, we obtain a ``weak'' solution concept  allowing for less regular initial data.
Nevertheless, for positive times all solutions are classical due to parabolic smoothing.

The structure of the paper is as follows: In Section \ref{fixtime} we consider the underlying two-phase boundary value problem for the 
Stokes equations~\eqref{probint}$_{1\--5}$ with fixed interface, and show that it is solved by the so-called hydrodynamic single-layer potential. 
We prove this by investigating its behavior near and on the interface (recovering results from \cite{Lad63} in our slightly different setting) and show that it vanishes in the far-field limit. 
This asymptotic result can be interpreted as nonoccurrence of the 2D Stokes paradoxon in our setting, 
which is essentially due to the fact that the curvature vector of the interface is a derivative (by arclength) of a vector that approaches a constant at infinity (see Eqn.~\eqref{STOE}).
In Section \ref{Sec:3} we first  rewrite our moving boundary problem as an evolution equation  for the function that parametrizes the interface between the fluids and announce our main result. 
The remainder of the section is dedicated to its proof. 
We linearize the evolution equation, and then establish its parabolic character (see  Proposition~\ref{T:GAP}). 
The localization procedure by which this is accomplished demands the main technical effort. 
Once parabolicity is established, the results follow from general facts on (fully) nonlinear problems of this type as given in \cite{L95}.

Throughout the paper, some longer proofs are deferred to appendices.

\section{\label{fixtime} The fixed time problem}
In this section we study the two-phase boundary value problem for the  Stokes equations~\eqref{probint}$_{1\--5}$ with fixed domains $\Omega^\pm$ and boundary $\partial\Omega^\pm:=\Gamma$ as defined by
\[\Omega^\pm:=\Omega_f^\pm:=\{(x_1,x_2)\in\mathbb{R}^2\,|\,x_2\gtrless f(x_1)\},
 \qquad \Gamma:=\Gamma_f:=\partial\Omega^\pm= \{(\xi,f(\xi))\,|\,\xi\in\mathbb{R}\}.\]
The function $f\in H^3(\mathbb{R})$ is  fixed. 
 Note that $\Gamma$ is the image of the  first coordinate axis under the diffeomorphism $\Xi:=\Xi_f:=({\rm id}_\mathbb{R},f)$.
 Further, let $\nu$ be the componentwise pull-back under~$\Xi$ of the unit normal on~$\Gamma$ exterior to~$\Omega^-$, i.e.
 \[\nu:=\tfrac{1}{\omega}(-f',1)^\top,
 \qquad \omega:=\omega_f:=(1+f'^2)^{1/2}.\]
 Let $\kappa:=\omega^{-3}f''\in H^{1}(\mathbb{R})$ be the pull-back under $\Xi$ of the curvature of $\Gamma$. 
 In view of
 \begin{align}\label{STOE}
 \kappa\nu=\omega^{-1}\big(\omega^{-1}(1,f')^\top\big)'
 \end{align}
 we will use the relation
 \be\label{derive}
 \kappa\nu=\omega^{-1}g',\qquad g:=g_f:=(g_1,g_2)^\top:=(\omega^{-1}-1,\omega^{-1}f')^\top=:(-\phi_1(f),\phi_2(f))^\top,
 \ee
 where
 \begin{align}\label{Phii}
 \phi_1(f)=\frac{{f'}^2}{\omega+\omega^2}\qquad\text{and}\qquad \phi_2(f)=\frac{f'}{\omega}
 \end{align}
 belong to $H^2(\R)$.

For any functions $z^\pm$ defined on $\Omega^\pm$, respectively, and having limits at some $ (\xi,f(\xi))\in\Gamma$ we will write
 \be\label{defjump}
 [z] (\xi,f(\xi)):=\lim_{ \Omega^+\ni x\to (\xi,f(\xi))}z^+(x)-
\lim_{ \Omega^-\ni x\to    (\xi,f(\xi))}z^-(x). 
\ee

We fix a common viscosity $\mu>0$ as well as a surface tension coefficient $\sigma>0$ and seek solutions
\[(v^\pm,q^\pm)\in \big({\rm C}^2(\Omega^\pm,\mathbb{R}^2)\cap {\rm C}^1(\overline{\Omega^\pm},\mathbb{R}^2)\big)
\times \big({\rm C}^1(\Omega^\pm)\cap {\rm C}(\overline{\Omega^\pm})\big)\]
to the two-phase boundary value problem 
\be\label{fixtimeeq}
\left.\begin{array}{rcll}
\mu\Delta v^\pm-\nabla q^\pm&=&0&\mbox{in $\Omega^\pm$,}\\
 \vdiv v^\pm&=&0&\mbox{in $\Omega^\pm$,}\\
 v^+&=&v^-&\mbox{on $\Gamma$,}\\
 {}[T(v,q)](\nu\circ\Xi^{-1})&=&-\sigma(\kappa\nu)\circ\Xi^{-1}&\mbox{on $\Gamma$,}\\
 (v^\pm,q^\pm) (x)&\to &0&\mbox{for $|x|\to\infty$,}
 \end{array}\right\}
\ee
 with $\Omega^\pm$ and $\Gamma$ as defined above.
Here $T(v,q)=(T_{ij}(v,q))_{1\leq i,\, j\leq 2}$ denotes the stress tensor  that is given by
\be\label{defT}T_{ij}(v,q):=-q\delta_{ij}+\mu(\partial_iv_j+\partial_j v_i).
\ee

The structure of the problem allows us to represent the solution as a hydrodynamic single-layer  potential \cite{Lad63}.
For this, we introduce  the fundamental solutions 
\[(\clu^k,\clp^k):\mathbb{R}^2\setminus\{0\}\longrightarrow\mathbb{R}^2\times\mathbb{R},\qquad  k=1,2,\]
to the Stokes equations in $\mathbb{R}^2$ by 
\begin{align*}
\clu^k&=(\clu^k_1, \clu^k_2)^\top,\\
\clu_j^k(y)&=-\frac{1}{4\pi\mu}\left(\delta_{jk}\ln\frac{1}{|y|}+
\frac{y_jy_k}{|y|^2}\right),\quad j=1,\,2,\\[1ex]
\clp^k(y)&=-\frac{1}{2\pi}\frac{y_k}{|y|^2},\quad y=(y_1,y_2)\in\RRM^2\setminus\{0\}.
\end{align*}
These functions solve (in distributional sense) the   Stokes equations
\begin{equation*}
\begin{array}{rcll}
\mu\Delta\clu^k-\nabla\clp^k&=&\delta_{ 0}e^k&\mbox{in $\mathcal{D}'(\mathbb{R}^2)$,}\\[1ex]
\vdiv\clu^k&=&0&\mbox{in $\mathcal{D}'(\mathbb{R}^2)$},
\end{array}
\end{equation*}
with  $e^1=(1,0)$ and $e^2=(0,1).$
Moreover,   differentiating the  fundamental
solutions with respect to $y_1$ and $y_2$ we get the following solutions to the homogeneous Stokes system in~${\mathbb{R}^2\setminus\{0\}}$:
 \be\label{stokesdiff}
 \begin{aligned}
     4\pi\partial_1\clu^1(y)&=\frac{1}{\mu |y|^4}
     \left(\begin{array}{c} y_1(y_1^2-y_2^2)\\y_2(y_1^2-y_2^2)\end{array}\right),&
     2\pi\partial_1\clp^1(y)&=\frac{y_1^2-y_2^2}{|y|^4},\\[1ex]
     4\pi\partial_2\clu^1(y)&=\frac{1}{\mu |y|^4}
     \left(\begin{array}{c} y_2(y_2^2+3y_1^2)\\y_1(y_2^2-y_1^2)\end{array}\right),&
     2\pi\partial_2\clp^1(y)&=\frac{2y_1y_2}{|y|^4},\\[1ex]
     4\pi\partial_1\clu^2(y)&=\frac{1}{\mu |y|^4}
     \left(\begin{array}{c} y_2(y_1^2-y_2^2)\\[1ex]
     y_1(y_1^2+3y_2^2)\end{array}\right),&
     2\pi\partial_1\clp^2(y)&=\frac{2y_1y_2}{|y|^4},\\[1ex]
      4\pi\partial_2\clu^2(y)&=\frac{1}{\mu |y|^4}
     \left(\begin{array}{c} y_1(y_2^2-y_1^2)\\y_2(y_2^2-y_1^2)\end{array}\right),&
     2\pi\partial_2\clp^2(y)&=\frac{y_2^2-y_1^2}{|y|^4}.
 \end{aligned}
 \ee

We are going to prove that the functions $(v^\pm,q^\pm):=(v,q)|_{\Omega^\pm}$ with 
\be\label{defvq}
\begin{aligned}
v(x)&:=\displaystyle\sigma\int_\mathbb{R}\partial_s\left(\clu^k(x-(s,f(s)))\right)g_k(s)\,ds,\\[1ex]
q( x)&:=\displaystyle-\sigma\int_\mathbb{R}\clp^k( x-(s,f(s)))g'_k(s)\,ds,\quad  x\in\RRM^2\setminus\Gamma,
\end{aligned}
\ee
and $g_1$, $g_2$ as defined in \eqref{derive} constitute the unique solution to \eqref{fixtimeeq}. First we check that the integrals exist.

 Observe that the kernels of the integral operators in \eqref{defvq} are smooth  with respect to~${x\in\mathbb{R}^2\setminus\Gamma}$.
 Moreover,
 $g_k'\in H^{1}(\mathbb{R})$ and 
 \[\text{$\clp^k( x-(s,f(s)))= O(s^{-1})$ \quad for $|s|\to\infty$,}\]
  so that the integrand  in \eqref{defvq}$_2$ belongs to $L_1(\mathbb{R})$.
  Furthermore, $g_k\to 0$ for $|s|\to\infty$, so we can use integration by parts to obtain
 \[q(x)=\sigma\int_\mathbb{R}\partial_s\big(\clp^k(x-(s,f(s))\big)
 g_k(s)\,ds.\]
  Recalling \eqref{stokesdiff}, we get
 \[\text{$\p_s(\clu^k(x-(s,f(s))))= O(s^{-1})$ \quad for $|s|\to\infty$,}\]
 and since $g_k\in H^{2}(\mathbb{R})$, it follows that also $v$ is well-defined.
 Altogether, we obtain the following representation for the velocity field and the pressure:
 \be\label{repvq}
 (v,q)^\top( x)=-\int_\mathbb{R}\partial_s [M( r)]g(s)\,ds
 =\int_\mathbb{R}(\partial_1 M( r)+f'(s)\partial_2 M( r))g(s)\,ds
 \ee
 for $ x\in\mathbb{R}^2\setminus\Gamma$, where 
 \begin{align}\label{defr}
     &r:=r(x,s):=x-(s,f(s)),\\[1ex]
     &M( y):=
     -\sigma\left(
     \begin{array}{cc}
     \clu^1&\clu^2\\[1ex]
     \clp^1&\clp^2
     \end{array}\right)(y),\quad y\in\R^2\setminus\{0\}.\nonumber
 \end{align}

\begin{thm}\label{T:1}
  Given $f\in H^3(\R)$,  Problem \eqref{fixtimeeq} has the unique solution $(v^\pm,q^\pm)$ given by~\eqref{defvq} or, equivalently,~\eqref{repvq}.
\end{thm} 
\begin{proof}
1. $(v^\pm,q^\pm)$ solves the Stokes equations:

 Denote the integrand in \eqref{repvq} by $J=[(x,s)\mapsto J(x,s)]$, 
 \[ J:(\mathbb{R}^2\setminus\Gamma)\times\mathbb{R}\longrightarrow \mathbb{R}^2\times\mathbb{R}.\]
 Any partial derivative  $\partial_x^\alpha J$
 can be dominated by an absolutely integrable function with respect to $s$, locally uniformly in $x$, so that differentiation with respect to $x$ and integration with respect to $s$ can be interchanged.
In particular, it follows that 
$$ (v^\pm,q^\pm)\in {\rm C}^\infty(\Omega^\pm,\mathbb{R}^2)\times {\rm C}^\infty(\Omega^\pm).$$
 As the columns of
 $\partial_1 M$, $\partial_2 M$ represent solutions to the homogeneous Stokes equations (cf.~\eqref{stokesdiff}),
  $(v^\pm,q^\pm)$ is also a solution to these equations on $\mathbb{R}^2\setminus\Gamma$.
 \pagebreak
 
 2.  Uniqueness:
  
 We have to show that  any solution 
 \begin{align*}
 (u^\pm,p^\pm)\in \big({\rm C}^2(\Omega^\pm,\mathbb{R}^2)\cap {\rm C}^1(\overline{\Omega^\pm},\mathbb{R}^2)\big)
\times \big({\rm C}^1(\Omega^\pm)\cap {\rm C}(\overline{\Omega^\pm})\big)
 \end{align*}
 to
 \[
\left.\begin{array}{rcll}
\mu\Delta u^\pm-\nabla p^\pm&=&0&\mbox{in $\Omega^\pm$,}\\
 \vdiv u^\pm&=&0&\mbox{in $\Omega^\pm$,}\\
 u^+&=&u^-&\mbox{on $\Gamma$,}\\
 {}[T(u,p)]{\widetilde\nu}&=&0&\mbox{on $\Gamma$,}\\
 (u,p) (x)&\to &0&\mbox{for $|x|\to\infty$,}
 \end{array}\right\}\]
 where $\widetilde\nu:=\nu\circ\Xi^{-1}$, is identically zero.
 Let $\widetilde\tau:=(\omega^{-1}(1,f')^\top)\circ\Xi^{-1}$ be the unit tangential vector field along $\Gamma$, oriented to the right.  Observe first that 
 \[T(u^\pm,p^\pm){\widetilde\nu}=\mu(\partial_{\widetilde\nu} u^\pm
 +\partial_{\widetilde\tau}(u_2,\,-u_1)^\top+{\widetilde\nu}\vdiv u^\pm)-p^\pm{\widetilde\nu},\]
 so under our assumptions
 \be\label{jumpT}
 [T(u,p)]{\widetilde\nu}=\mu[\partial_{\widetilde\nu} u]-[p]{\widetilde\nu}=0.
 \ee
We now define
 \[(V,Q):={\bf 1}_{\Omega^+}(u^+,p^+)+{\bf 1}_{\Omega^-}(u^-,p^-)\in L_{1,\rm loc}(\mathbb{R}^2,\mathbb{R}^2\times\mathbb{R}).\]
 Taking distributional derivatives and using the continuity of $u$ across $\Gamma$ yields
 \begin{align*}
     \Delta V_i&={\bf 1}_{\Omega^+}\Delta u_i^++{\bf 1}_{\Omega^-}\Delta u_i^-+[\partial_{\wt \nu} u_i]\delta_\Gamma,\\
\partial_i Q&={\bf 1}_{\Omega^+}\partial_i p^+
+{\bf 1}_{\Omega^-}\partial_i p^-+[p]{\widetilde\nu}_i\delta_\Gamma,\\
\vdiv V&=0,
 \end{align*}
  where, given $a\in L_{1,\rm loc}(\Gamma)$, the distribution $a\delta_{\Gamma}$ is defined by
 $$\langle a\delta_\Gamma|\phi\rangle:=\int_\Gamma a\phi\,d\Gamma,\qquad \phi\in C_0^\infty(\mathbb{R}^2).$$
  So, from this and \eqref{jumpT} we get
 \[\mu\Delta V-\nabla Q=0\qquad\text{in $\mathcal{D}'(\mathbb{R}^2)$}.\]
  In particular, taking the divergence of this equation yields $\Delta Q=0$, i.e. $Q$ is a harmonic function on the full space $\mathbb{R}^2$, and the asymptotic condition implies $Q=0$ via Liouville's theorem. This implies in turn that $V_1$ and $V_2$ are harmonic, and are therefore zero by the same argument.
 
 3.  The behavior of $(v^\pm,q^\pm)$ near $\Gamma$ is addressed in Appendix~\ref{appA}. In particular, it is shown that
 $(v^\pm,q^\pm)\in   {\rm C}^1(\overline{\Omega^\pm},\mathbb{R}^2) \times   {\rm C}(\overline{\Omega^\pm}) $ satisfies   Eqns. \eqref{fixtimeeq}$_{3-4}$.

 4.  The far-field boundary condition   \eqref{fixtimeeq}$_{5}$ is established in Appendix \ref{appB}.
\end{proof}

\pagebreak

 \section{The  evolution problem}\label{Sec:3}
 In the first part of this section we introduce some notation which is then used to recast the Stokes problem \eqref{probint} as an evolution problem for $f$ only, see \eqref{NNEP} below.
In the second part we  establish our  main result stated in Theorem~\ref{MT1}.

\subsection{\label{subsec31} A class of singular integral operators}
 We first introduce a class of multilinear singular integral operators which are  needed in the second part of this section.
Given~${n,\,m\in\N}$ and  Lipschitz continuous  functions $a_1,\ldots, a_{m},\, b_1, \ldots, b_n:\mathbb{R}\longrightarrow\mathbb{R}$,  denote by $B_{n,m}$ the singular integral   operator 
\begin{equation}\label{BNM}
B_{n,m}(a_1,\ldots, a_m)[b_1,\ldots,b_n,h]( \xi):=\PV\int_\mathbb{R}  \frac{h( \xi- \eta)}{ \eta}
\cfrac{\prod_{i=1}^{n}\big(\delta_{[\xi,\eta]} b_i / \eta\big)}{\prod_{i=1}^{m}\big[1+\big(\delta_{[ \xi, \eta]}  a_i / \eta\big)^2\big]}\, d \eta,
\end{equation}
where $\delta_{[ \xi, \eta]}u:=u(\xi)-u(\xi- \eta)$, and for brevity
\be\label{defB0}
B^0_{n,m}(f)[h]:=B_{n,m}(f,\ldots f)[f,\ldots,f,h]
\ee
(with the appropriate number of identical arguments $f$ filled in).
Here $\PV$ denotes the principle value.
Below we write ${\rm C}^{1-}(X,Y)$ for the space of locally Lipschitz maps from $X$ to $Y$.
 Furthermore, given Banach spaces $X$ and $Y$, we let $\kL^k_{\rm sym}(X,Y)$ denote the space of $k$-linear, bounded symmetric maps $A:\;X^k\longrightarrow Y$.
The following properties are extensively used in our analysis.

\begin{lemma}\label{L:MP0}

$$
$$

\vspace{-0.5cm}
\begin{itemize}
\item[(i)] Given  Lipschitz continuous  functions $a_1,\ldots, a_{m},\, b_1, \ldots, b_n:\mathbb{R}\longrightarrow\mathbb{R}$,  there exists a constant~$C$ depending only 
on $n,\, m$ and $\max_{i=1,\ldots, m}\|a_i'\|_{\infty}$, such that
 $$\|B_{n,m}(a_1,\ldots, a_m)[b_1,\ldots,b_n,\,\cdot\,]\|_{\kL(L_2(\mathbb{R}))}\leq C\prod_{i=1}^{n} \|b_i'\|_{\infty}.$$

 Moreover,   $B_{n,m}\in {\rm C}^{1-}((W^1_\infty(\mathbb{R}))^{m},\kL^n_{\rm sym}(W^1_\infty(\mathbb{R}), \kL(L_2(\mathbb{R})))).$\\[-1ex]
 \item[(ii)] Given $s\in(3/2 ,2)$, there exists a constant C,
  depending only on $n,\, m,\, s$,  and $\max_{1\leq i\leq m}\|a_i\|_{H^s},$ such that
\begin{align*} 
\| B_{n,m}(a_1,\ldots, a_{m})[b_1,\ldots, b_n,h]\|_{H^{s-1}}\leq C \|h\|_{H^{s-1}}\prod_{i=1}^{n}\|b_i\|_{H^{s}}
\end{align*}
for all $a_1,\ldots, a_m,\, b_1,\ldots, b_n\in H^s(\mathbb{R})$ and $h\in H^{s-1}(\mathbb{R}). $

Moreover,   $  B_{n,m}\in {\rm C}^{1-}((H^s(\mathbb{R}))^m, \kL^n_{\rm sym}(H^s(\mathbb{R}), \kL(H^{s-1}(\mathbb{R})))).$ \\[-1ex]

\item[(iii)]  Let  $n\geq 1$ and $3/2<s'<s<2$  be given. 
  There exists a constant  $C$, depending only on $n,\, m$, $s$, $s'$,  and $\max_{1\leq i\leq m}\|a_i\|_{H^s}$, such that
\begin{equation*} 
\begin{aligned} 
&\| B_{n,m}(a_1,\ldots, a_{m})[b_1,\ldots, b_n,h] -h B_{n-1,m}(a_1,\ldots, a_{m})[b_2,\ldots, b_n,b_1']\|_{H^{s-1}}\\[1ex]
&\hspace{3cm}\leq C \|b_1\|_{H^{s'}}\|h\|_{H^{s-1}}\prod_{i=2}^{n}\|b_i\|_{H^s}
\end{aligned}
\end{equation*}
for all $a_1,\ldots, a_m,\, b_1,\ldots, b_n\in H^s(\mathbb{R})$ and $h\in H^{s-1}(\mathbb{R}).$
\end{itemize}
\end{lemma}
\begin{proof}
The claim (i) is established in \cite[Lemma 3.1]{MBV18}, while the properties (ii) and (iii) are proven in \cite[Lemmas~2.5 and~2.6]{AbMa20x}.
\end{proof}

\subsection{Formulation of the evolution equation and the main result}

In view of Theorem~\ref{T:1} we may recast the two-phase Stokes moving boundary problem \eqref{probint} as a nonlinear and nonlocal evolution problem of the form
\be\label{evol0}
\frac{df}{dt}(t)=-f'(t)v_1|_\Gamma\circ\Xi+v_2|_\Gamma\circ\Xi
\ee
with $v=(v_1,v_2)$  given by~\eqref{defvq}.  As shown in Lemma \ref{neargamma}, the extension of $v$ to $\Gamma$ exists and is given by
\[
\begin{aligned}
 \hspace{-0.2cm} v_1|_\Gamma\circ\Xi&=\frac{(B_{2,2}^0(f)-B_{0,2}^0(f))[g_1-f'g_2]-B_{1,2}^0(f)[3f'g_1+g_2] -B_{3,2}^0(f)[f'g_1-g_2]}{4\pi\mu \sigma^{-1}},\\[1ex]
  \hspace{-0.2cm} v_2|_\Gamma\circ\Xi&=\frac{B_{0,2}^0(f)[f'g_1-g_2]+(B_{3,2}^0(f)-B_{1,2}^0(f))[g_1-f'g_2]-B_{2,2}^0(f)[f'g_1+3g_2]}{4\pi\mu  \sigma^{-1}},
\end{aligned}\]
with $g_j$ from \eqref{derive}.

Therefore \eqref{evol0} can be written as an evolution equation for $f$ in the form
\begin{align}\label{NNEP}
\frac{df}{dt}(t)=\Psi(f(t)),\quad t\geq0,\qquad f(0)=f_0,
\end{align}
with
\begin{align}\label{PHI}
\Psi(f)&:=-\frac{\sigma}{4\pi\mu}f'\Psi_1(f)+\frac{\sigma}{4\pi\mu}\Psi_2(f),
\end{align}
 where the nonlinear operators $\Psi_j$, $j=1,\,2$, are defined by
\begin{align*}
\Psi_1(f)&:=(B^0_{0,2}(f)-B^0_{2,2}(f))[\phi_1+f'\phi_2]+B^0_{1,2}(f)[3f'\phi_1-\phi_2]+B^0_{3,2}(f)[f'\phi_1+\phi_2],\\[1ex]
\Psi_2(f)&:=-B^0_{0,2}(f)[f'\phi_1+\phi_2]+(B^0_{1,2}(f)-B^0_{3,2}(f))[\phi_1+f'\phi_2]+B^0_{2,2}(f)[f'\phi_1-3\phi_2].
\end{align*}
We recall from \eqref{Phii} the shorthand notations
\[
\phi_1(f)=\frac{f'^2}{\sqrt{1+f'^2}+1+f'^2}\quad\text{and}\quad \phi_2(f)=\frac{f'}{\sqrt{1+f'^2}}.
\]

 The following theorem contains the main results of this paper.
\begin{thm}\label{MT1} Let  $s\in(3/2,2) $ be given.
Then, the following  statements hold true:
\begin{itemize}
\item[(i)]  {\em (Well-posedness)}  Given $f_0\in H^{s}(\mathbb{R})$, there exists a unique maximal solution 
\[f=f(\cdot;f_0)\in {\rm C}([0,T_+),  H^{s}(\mathbb{R}))\cap {\rm C}^1([0,T_+), H^{s-1}(\mathbb{R})),\]
where $T_+=T_+(f_0)\in (0,\infty]$, to \eqref{NNEP}. 
Moreover, $[(t,f_0)\mapsto  f(t;f_0)]$ defines a  semiflow on $H^{s}(\mathbb{R})$.\\[-2.2ex]
\item[(ii)]  {\em (Parabolic smoothing)} 
\begin{itemize}
\item[(iia)]  The map $[(t, \xi)\mapsto  f(t)(\xi)]:(0,T_+)\times\mathbb{R}\longrightarrow\mathbb{R}$ is a ${\rm C}^\infty$-function. 
\item[(iib)] For any $k\in\N$, we have $f\in {\rm C}^\infty ((0,T_+), H^k(\mathbb{R})).$\\[-2ex]
\end{itemize} 
\item[(iii)]  {\em (Global existence)} If 
$$\sup_{[0,T]\cap [0,T_+(f_0))} \|f(t)\|_{H^s}<\infty$$
for each $T>0$, then $T_+(f_0)=+\infty.$
\end{itemize} 
\end{thm}

\begin{rem}
If $f$ is a solution to \eqref{NNEP}, then, given $\lambda>0$, also
\[
f_\lambda(t,\xi):=\lambda^{-1}f(\lambda t,\lambda  \xi),
\]
is a solution to \eqref{NNEP} (with initial datum $\lambda^{-1}f(0,\cdot)$). 
This property identifies~$H^{3/2}(\mathbb{R})$ as a critical space for the evolution problem \eqref{NNEP}.
Therefore, our result in Theorem~\ref{MT1}   covers all subcritical spaces. 
\end{rem}

\begin{rem}
We expect the solutions to be even analytic in space and time away from~$t=0$. 
However, we prefer to formulate and prove our result in the ${\rm C}^\infty$-class, refraining from the considerable technicalities needed for a proof of the analytic counterpart of Lemma {\rm\ref{L:REG}} below 
(see \cite[Proposition  5.1]{MBV19} for a related  analyticity result). 
\end{rem}

In order to study the mapping properties of the operator $\Psi$ we need the following lemmas.

\begin{lemma}\label{L:MP1}
Given $s\in(3/2,2)$, we have 
$\phi_i\in {\rm C}^\infty(H^{s}(\mathbb{R}), H^{s-1}(\mathbb{R}))$, $i=1,\, 2$.
Moreover, given~${f_0\in H^s(\mathbb{R}),}$ the Fr\'echet derivative $\p\phi_i(f_0)$ is given by 
$$\p\phi_i(f_0)=a_i(f_0)\frac{d}{d\xi},$$
with $a_i$ defined by
\[
a_1(f_0):=\frac{f_0'(2+f_0'^2+2\sqrt{1+f_0'^2})}{\sqrt{1+f_0'^2}(\sqrt{1+f_0'^2}+1+f_0'^2)^2}\qquad\text{and}\qquad a_2(f_0):=\frac{1}{(1+f_0'^2)^{3/2}}.
\]
\end{lemma}
\begin{proof} For the smoothness result we refer to Lemma \ref{phismooth} in Appendix \ref{appsmooth}.
The representations for the derivatives  $\p\phi_i(f_0)$ follow from straightforward calculations. 
\end{proof}

\begin{lemma}\label{L:REG}
Given $s\in(3/2,2)$, we have 
\begin{align*}
\Psi\in {\rm C}^\infty(H^{s}(\mathbb{R}), H^{s-1}(\mathbb{R})).
\end{align*}
\end{lemma}
\begin{proof} 
The claim follows from Lemma \ref{L:MP1} and Corollary \ref{Bsmooth}.
\end{proof}

 For two Banach spaces $X_0,X_1$ with dense embedding $X_1\hookrightarrow X_0$, let $\kH(X_1,X_0)$ denote the set of operators $A\in\kL(X_1,X_0)$ 
 such that $-A$ generates a strongly continuous and analytic semigroup of operators on $X_0$.

In order to establish our main result in Theorem \ref{MT1} we next prove a generator property for
the Fr\'echet derivative $\p\Psi(f_0)\in\kL(H^s(\mathbb{R}), H^{s-1}(\mathbb{R}))$  
 which  identifies \eqref{NNEP} as a  nonlinear evolution  problem of parabolic type.

\begin{prop}\label{T:GAP}
Given $f_0\in H^s(\mathbb{R})$, we have
\begin{align}\label{eq:GP}
-\p\Psi(f_0)\in\mathcal{H}(H^s(\mathbb{R}), H^{s-1}(\mathbb{R})).
\end{align}
\end{prop}

The subsequent analysis is devoted to the proof of  Proposition~\ref{T:GAP}.
To start, we fix a function~${f_0\in H^s(\mathbb{R})}$  and  $s'\in(3/2,s)$, and we note that
\[
\p\Psi(f_0)[f]=-\frac{\sigma}{4\pi\mu}f'\Psi_1(f_0)-\frac{\sigma}{4\pi\mu} f_0'\p\Psi_1(f_0)[f]+\frac{\sigma}{4\pi\mu}\p\Psi_2(f_0)[f], \qquad f\in H^s(\mathbb{R}).
\]

To calculate the derivatives of $\Psi_i$ we use Lemma \ref{Bdiff} to get
\[\partial B_{n,2}^0(f_0)[f][h]=n B_{n,2}(f_0,f_0)[f,f_0,\ldots f_0,h]-4B_{n+2,3}(f_0,f_0,f_0)[f,f_0,\ldots,f_0,h]\]
and Lemma \ref{L:MP0}~(iii) to rewrite this for $n\geq0$ as
\begin{align*}
\partial B_{n,2}^0(f_0)[f][h]&=h\big(
nB^0_{n-1,2}(f_0)[f']-4B^0_{n+1,3}(f_0)[f']\big)+R_n[f]\\
&=h\big(n B^0_{n-1,3}(f_0)[f']+(n-4)B^0_{n+1,3}(f_0)[f']\big)+R_n[f],
\end{align*}
where  $n B^0_{n-1,3}(f_0):=0$ for $n=0$ and
\[\|R_nf\|_{H^{s-1}}\leq C\|h\|_{H^{s-1}}\|f\|_{H^{s'}}. \]
 The constant $C$ is independent of $f\in H^{s}(\R)$ and $h\in H^{s-1}(\R).$
From this and the definition of $\Psi_i,$ $i=1,\, 2$, we get
\begin{equation}\label{pPsi1}
\begin{aligned}
\p\Psi_i(f_0)[f]&=T_{i,1}(f_0)[f]+T_{i,2}(f_0)[f]+T_{i,{\rm lot}}(f_0)[f],\qquad i=1,\,2,
\end{aligned}
\end{equation}
where
\begin{equation}\label{pPsi2-new}
\begin{aligned}
T_{1,1}(f_0)[f]:=&(B^0_{0,2}-B^0_{2,2})[(a_1+\phi_2+f_0'a_2)f']+B^0_{1,2}[(3(\phi_1+f_0'a_1)-a_2)f']\\[1ex]
&+B_{3,2}^0[(\phi_1+f_0'a_1+a_2)f'],\\[1ex]
T_{1,2}(f_0)[f]:=&\phi_1(3f_0'B_{0,3}^0-6B_{1,3}^0-6f_0'B_{2,3}^0+2B_{3,3}^0-f_0'B_{4,3}^0)[f']\\[1ex]
&+\phi_2(-B_{0,3}^0-6f_0'B_{1,3}^0+6B_{2,3}^0+2f_0'B_{3,3}^0-B_{4,3}^0)[f'],\\[1ex]
T_{2,1}(f_0)[f]:=&-B_{0,2}^0[(\phi_1+f_0'a_1+a_2)f']+(B^0_{1,2}-B^0_{3,2})[(a_1+\phi_2+f_0'a_2)f']\\[1ex]
&+B^0_{2,2}[(\phi_1+f'_0a_1-3a_2)f'],\\[1ex]
T_{2,2}(f_0)[f]:=&\phi_1(B_{0,3}^0+6f_0'B_{1,3}^0-6B_{2,3}^0-2f_0'B_{3,3}^0+B_{4,3}^0)[f']\\[1ex]
&+\phi_2(f_0'B_{0,3}^0-2B_{1,3}^0-6f_0'B_{2,3}^0+6B_{3,3}^0+f_0'B_{4,3}^0)[f'],
\end{aligned}
\end{equation}
with shortened notation $a_i=a_i(f_0)$, $\phi_i=\phi_i(f_0)$,
$B_{n,m}^0=B_{n,m}^0(f_0)$, and
\begin{align}\label{pPsi3}
\|T_{i,{\rm lot}}(f_0)[f]\|_{H^{s-1}}\leq C\max\{\|\phi_1\|_{H^{s-1}},\|\phi_2\|_{H^{s-1}}\}\|f\|_{H^{s'}}\leq C\|f\|_{H^{s'}},\quad f\in H^s(\R).
\end{align}
Having computed the derivative  $\p\Psi(f_0)$, it remains to  establish \eqref{eq:GP}, which is achieved via  a localization procedure.
To proceed, we fix for each $\e\in(0,1)$ a so-called finite~$\e$-localization family, that is a set 
\[\{(\pi_j^\e,\xi_j^\e)\,:\, -N+1\leq j\leq N\}\subset{\rm C}^\infty(\R,[0,1])\times\R\]
 such that
\begin{align*}
\bullet\,\,\,\, \,\,&\pi_j^\eps\in C^\infty(\mathbb{R},[0,1]);\\[1ex]
\bullet\,\,\,\, \,\,  & \text{$ \supp \pi_j^\e $ is an interval of length $\e$ for all $|j|\leq N-1$, $ \supp \pi_{N}^\e\subset(-\infty,-1/\e]\cup [1/\e,\infty)$;} \\[1ex]
\bullet\,\,\,\, \,\, &\text{ $ \pi_j^\e\cdot  \pi_l^\e=0$ if $[|j-l|\geq2, \max\{|j|, |l|\}\leq N-1]$ or $[|l|\leq N-2, j=N];$} \\[1ex]
\bullet\,\,\,\, \,\, &\text{ $\sum_{j=-N+1}^N(\pi_j^\e)^2=1;$} \\[1ex]
 \bullet\,\,\,\, \,\, &\text{$\|(\pi_j^\e)^{(k)}\|_\infty\leq C\e^{-k}$ for all $ k\in\N, -N+1\leq j\leq N$;} \\[1ex]
 \bullet\,\,\,\, \,\, &\xi^\eps_j\in\supp\pi_j^\eps,\; |j|\leq N-1. 
 \end{align*} 
 The real number $\xi_N^\e$ plays no role in the analysis below.
To each  finite $\e$-localization family we associate  a second family   
$$\{\chi_j^\e\,:\, -N+1\leq j\leq N\}\subset {\rm C}^\infty(\R,[0,1])$$ with the following properties
\begin{align*}
\bullet\,\,\,\, \,\,  &\text{$\chi_j^\e=1$ on $\supp \pi_j^\e$;} \\[1ex]
\bullet\,\,\,\, \,\,  &\text{$\supp \chi_j^\e$ is an interval  of length $3\e$ and with the same midpoint as $ \supp \pi_j^\e$, $|j|\leq N-1$;} \\[1ex]
\bullet\,\,\,\, \,\,  &\text{$\supp\chi_N^\e\subset [|x|\geq 1/\e-\e]$.}  
\end{align*} 
 Using the $\e$-localization family we define norms on $H^s(\mathbb{R}),$ $s\geq 0$ that are  equivalent to the standard  norm.
Indeed,  it is not difficult to prove that, given $\e\in(0,1)$ and $s\geq0$, there exists a constant $c=c(\e,s)\in(0,1)$ such that
\begin{align}\label{EQNO}
c\|f\|_{H^s}\leq \sum_{j=-N+1}^N\|\pi_j^\e f\|_{H^s}\leq c^{-1}\|f\|_{H^s},\qquad f\in H^{s}(\mathbb{R}).
\end{align}

  To show \eqref{eq:GP} we use a homotopy argument.
   For this we consider the continuous path 
\[\Phi:[0,1]\longrightarrow\kL(H^{s}(\mathbb{R}), H^{s-1}(\mathbb{R}))\]
defined by
\begin{align*}
\Phi(\tau):=-\frac{\tau\sigma}{4\pi\mu}\Psi_1(f_0)\frac{d}{d\xi}-\frac{\sigma\tau}{4\pi\mu} f_0'\p\Psi_1(\tau f_0)+\frac{\sigma}{4\pi\mu}\p\Psi_2(\tau f_0),\qquad \tau\in[0,1].
\end{align*}
We next locally approximate the operator  $\Phi(\tau)$,  $\tau\in[0,1]$,   by  certain Fourier multipliers~$\bA_{j,\tau}$. 
It is worth emphasizing that $\Phi(1)=\p\Psi(f_0)$, while $\Phi(0)$ is the Fourier multiplier given by 
\[
\Phi(0)=-\frac{\sigma}{4 \mu} H\circ \frac{d}{d\xi}=-\frac{\sigma}{4 \mu}\Big(- \frac{d^2}{d\xi^2}\Big)^{1/2},
\]
with $H$ denoting the Hilbert transform.
The  homotopy $\Phi$  will be used to conclude invertibility of $\lambda-\Phi(1)$ from $\lambda-\Phi(0)$  for sufficiently large $\lambda$.  
We also point out the estimate
\begin{align}\label{MES}
\|gh\|_{H^{s-1}}\leq 2(\|g\|_\infty\|h\|_{H^{s-1}}+\|h\|_\infty\|g\|_{H^{s-1}})\qquad\text{for $g,\, h\in H^{s-1}(\mathbb{R})$, $s\in(3/2,2)$,}
\end{align} 
which is used several times in the arguments that follow.

 From a technical point of view, the following proposition is at the core of the proof of Proposition \ref{T:GAP}. 
 It provides estimates for the errors introduced by replacing the operator~$\Phi(\tau)$ by the localizations ~$\bA_{j,\tau}$.
 In fact, this amounts to the well-known ``freezing of coefficients'' in the context of our nonlocal operators.

\begin{prop}\label{T:AP} 
Let $\gamma>0$ be given and fix $s'\in (3/2,s)$. 
Then, there exist $\e\in(0,1)$,  a constant $K=K(\e)$, 
and   bounded operators 
$$
\bA_{j,\tau}\in\kL(H^s(\mathbb{R}), H^{s-1}(\mathbb{R})), \qquad\text{$j\in\{-N+1,\ldots,N\}$, $\tau\in[0,1]$,} 
$$
 such that 
 \begin{equation}\label{D1}
  \|\pi_j^\e \Phi(\tau) [f]-\bA_{j,\tau}[\pi^\e_j f]\|_{H^{s-1}}\leq \gamma \|\pi_j^\e f\|_{H^s}+K\|  f\|_{H^{s'}}
 \end{equation}
 for all $ j\in\{-N+1,\ldots,N\}$, $\tau\in[0,1],$  and  $f\in H^s(\mathbb{R})$. 
 The operators $\bA_{j,\tau}$ are defined  by 
  \begin{align*} 
 \bA_{j,\tau }:=- \alpha_\tau(\xi_j^\e) \Big(-\frac{d^2}{d\xi^2}\Big)^{1/2}+\beta_\tau ( \xi_j^\e)\frac{d}{d\xi}, \quad |j|\leq N-1, \qquad\bA_{N,\tau }:= -  \frac{\sigma}{4\mu} \Big(-\frac{d^2}{d\xi^2}\Big)^{1/2},
 \end{align*}
  with functions $\alpha_\tau,\, \beta_\tau$ given by
 \begin{align*}
 \alpha_\tau:=\frac{\sigma}{4\mu}[a_2(\tau f_0)+\tau f_0'a_1(\tau f_0)], \qquad  \beta_\tau:= -\frac{\sigma\tau }{4\pi\mu}\Psi_1(f_0).   
 \end{align*}
\end{prop}
\begin{proof} Let $\eps\in(0,1)$.
In the following we denote by  $K$ constants that may  depend  on $\e $.

The estimate \eqref{pPsi3} implies
\begin{align}\label{EST:0}
\|\pi_j^\e f_0'T_{1,\rm lot}(\tau f_0)[f]\|_{H^{s-1}}+\|\pi_j^\e T_{2,\rm lot}(\tau f_0)[f]\|_{H^{s-1}}\leq K\|f\|_{H^{s'}},\quad -N+1\leq j\leq N.
\end{align}
Next we consider the operators $\Psi_1(f_0)(d/d\xi)$, $f_0'T_{1,j}(\tau f_0)$,   and $T_{2,j}(\tau f_0)$, ${j=1,\, 2}$, which we approximate successively.
\medskip

\noindent{\em Step 1.} We first consider the term $f'\Psi_1(f_0)$. 
 Using $\Psi_1(f_0)\in H^{s-1}(\mathbb{R}) \hookrightarrow {\rm C}^{s-3/2}(\mathbb{R})$ and  the property $\chi_j^\e\pi_j^\e=\pi_j^\e$ together with \eqref{MES}, we get
\begin{equation}\label{EST:1a}
\begin{aligned}
&\hspace{-0.5cm}\|\pi_j^\e f'\Psi_1(f_0)-\Psi_1(f_0)(\xi_j^\e)(\pi_j^\e f)'\|_{H^{s-1}}\\[1ex]
&\leq  2\|\chi_j^\e (\Psi_1(f_0)-\Psi_1(f_0)(\xi_j^\e))\|_\infty\|(\pi_j^\e f)'\|_{H^{s-1}}+K\|f\|_{H^{s'}}\\[1ex]
&\leq \frac{\gamma}{3} \|\pi_j^\e f\|_{H^{s}}+K\|f\|_{H^{s'}}
\end{aligned}
\end{equation}
 for $|j|\leq N-1$,  provided that $\e$ is sufficiently small.
Furthermore, since $\Psi_1(f_0)$ vanishes at infinity, we have
\begin{equation}\label{EST:1b}
\begin{aligned}
\|\pi_N^\e f'\Psi_1(f_0)\|_{H^{s-1}}&\leq 2\|\chi_N^\e \Psi_1(f_0)\|_\infty\|(\pi_N^\e f)'\|_{H^{s-1}}+K\|f\|_{H^{s'}}\\[1ex]
&\leq\frac{\gamma}{3} \|\pi_N^\e f\|_{H^{s}}+K\|f\|_{H^{s'}}, 
\end{aligned}
\end{equation}
 provided that $\e$ is sufficiently small.\medskip

\noindent{\em Step 2.} We now consider the operators $f_0'T_{1,2}(\tau f_0)$   and $T_{2,2}(\tau f_0)$.
Repeated use of  Lemma~\ref{L:B1} and Lemma~\ref{L:B2} yields
\begin{equation}\label{EST:2a}
\begin{aligned}
&\|\tau \pi_j^\e f_0'T_{1,2}(\tau f_0)[f]-a_{1,\tau}(\xi_j^\e)B_{0,0}[(\pi_j^\e f)']\|_{H^{s-1}}
\leq  \frac{\gamma}{6} \|\pi_j^\e f\|_{H^{s}}+K\|f\|_{H^{s'}},\\[1ex]
&\|\pi_j^\e  T_{2,2}(\tau f_0)[f]- a_{2,\tau}(\xi_j^\e)B_{0,0}[(\pi_j^\e f)']\|_{H^{s-1}}\leq \frac{\gamma}{6} \|\pi_j^\e f\|_{H^{s}}+K\|f\|_{H^{s'}}
\end{aligned} 
\end{equation} 
for $|j|\leq N-1$, and 
\begin{align}\label{EST:2b}
&\|\tau\pi_N^\e f_0'T_{1,2}(\tau f_0)[f]\|_{H^{s-1}}+\| \pi_N^\e T_{2,2}(\tau f_0)[f]\|_{H^{s-1}}\leq \frac{\gamma}{3} \|\pi_N^\e f\|_{H^{s}}+K\|f\|_{H^{s'}},
\end{align}
provided that $\e$ is sufficiently small.
The functions  $a_{i,\tau}$, $i=1,\, 2$, $\tau\in[0,1]$, are given~by
\begin{align*}
 a_{1,\tau}&:=-\phi_1(\tau f_0)\frac{(\tau f_0')^4+3(\tau f_0')^2}{(1+(\tau f_0')^2)^2}
+\phi_2(\tau f_0)\frac{(\tau f_0')^3-\tau f_0'}{(1+(\tau f_0')^2)^2},\\[1ex]
 a_{2,\tau}&:=-\phi_1(\tau f_0)\frac{(\tau f_0')^2-1}{(1+(\tau f_0')^2)^2}
+\phi_2(\tau f_0)\frac{(\tau f_0')^3-\tau f_0'}{(1+(\tau f_0')^2)^2}.
\end{align*}

\noindent{\em Step 3.} We now consider the operators $f_0'T_{1,1}(\tau f_0)$   and $T_{2,1}(\tau f_0)$.
We first observe  that $a_1(f_0),\, a_2(f_0)-1,\,\phi_i(f_0)\in H^{s-1}(\mathbb{R})$, $i=1,\, 2.$
Repeated use of  Lemmas~\ref{L:B1}--\ref{L:B5} leads to 
\begin{equation}\label{EST:3a}
\begin{aligned}
&\|\tau \pi_j^\e f_0'T_{1,1}(\tau f_0)[f]-a_{3,\tau}(\xi_j^\e)B_{0,0}[(\pi_j^\e f)']\|_{H^{s-1}}\leq  \frac{\gamma}{6} \|\pi_j^\e f\|_{H^{s}}+K\|f\|_{H^{s'}},\\[1ex]
&\|\pi_j^\e T_{2,1}(\tau f_0)[f]-a_{4,\tau}(\xi_j^\e)B_{0,0}[(\pi_j^\e f)']\|_{H^{s-1}}\leq \frac{\gamma}{6} \|\pi_j^\e f\|_{H^{s}}+K\|f\|_{H^{s'}}
\end{aligned} 
\end{equation} 
for $|j|\leq N-1$  and 
\begin{equation}\label{EST:3b}
\begin{aligned}
&\|\tau \pi_N^\e f_0'T_{1,1}(\tau f_0)[f]\|_{H^{s-1}}+\| \pi_N^\e T_{2,1}(\tau f_0)[f]+B_{0,0}[(\pi_N^\e f)']\|_{H^{s-1}}\\[+1ex]
&\hspace{3cm}\leq \frac{\gamma}{3} \|\pi_N^\e f\|_{H^{s}}+K\|f\|_{H^{s'}},
\end{aligned}
\end{equation}
provided that $\e$ is sufficiently small.
The functions  $a_{i,\tau}$, $i=3,\, 4$,  $\tau\in[0,1]$, are given~by
\begin{align*}
 a_{3,\tau}:=a_1(\tau f_0)\tau f_0'- a_{1,\tau}\quad\text{and}\quad a_{4,\tau}:=-a_2(\tau f_0)- a_{2,\tau}.
\end{align*}
Gathering \eqref{EST:0}-\eqref{EST:3b}, we conclude that \eqref{D1} holds true and the proof is complete.
\end{proof}

Making use of $\Psi_1(f_0)\in H^{s-1}(\mathbb{R})$ and recalling the definition of the functions~$a_i$ in Lemma~\ref{L:MP1},
we conclude there exists $\eta\in(0,1)$ such that 
\[
\eta\leq \alpha_\tau\leq \frac{1}{\eta}\quad\text{and}\quad |\beta_\tau|\leq \frac{1}{\eta},\qquad\tau\in[0,1].
\]
Given  $\alpha\in[\eta,1/\eta]$ and $|\beta|\leq 1/\eta$, we now introduce  the Fourier multipliers
\begin{align*} 
 \bA_{\alpha,\beta}:=- \alpha\Big(-\frac{d^2}{d\xi^2}\Big)^{1/2}+\beta\frac{d}{d\xi}\in\kL(H^{s}(\mathbb{R}), H^{s-1}(\mathbb{R})).  
 \end{align*}
 
 For two Banach spaces $X$, $Y$, let ${\rm Isom}(X,Y)$ denote the set of (linear and topological)
 isomorphism from $X$ to $Y$.
 
 It is a matter of direct computations using Fourier representations, see e.g. \cite[Proposition~4.3]{MBV19}, to find a constant~${\kappa_0\geq 1}$   such that
 \begin{align}
\bullet &\quad \lambda-\bA_{\alpha,\beta}\in {\rm Isom}(H^s(\mathbb{R}),H^{s-1}(\mathbb{R})),\qquad  \re\lambda\geq 1,\label{L:FM1}\\[1ex]
\bullet &\quad  \kappa_0\|(\lambda-\bA_{\alpha,\beta})[f]\|_{H^{s-1}}\geq |\lambda|\cdot\|f\|_{H^{s-1} }+\|f\|_{H^s}, \qquad f\in H^{s}(\mathbb{R}),\, \re\lambda\geq 1\label{L:FM2}.
\end{align}

 Proposition~\ref{T:AP} and the relations \eqref{L:FM1} and \eqref{L:FM2} combined enable us to establish the result announced in Proposition~\ref{T:GAP}.

\begin{proof}[Proof of  Proposition~\ref{T:GAP}]
Let $s'\in(3/2,s)$ and  let  $\kappa_0\geq1$ be the constant   in~\eqref{L:FM2}. 
  Proposition~\ref{T:AP} with ${\gamma:=1/2\kappa_0}$ implies that there are $\eps\in(0,1)$, a constant $K=K(\e)>0$
 and  bounded operators $\bA_{j,\tau}\in\kL(H^s(\mathbb{R}), H^{s-1}(\mathbb{R}))$, $ -N+1\leq j\leq N$ and $\tau\in[0,1],$ satisfying 
 \begin{equation*} 
  2\kappa_0\|\pi_j^\e\Phi(\tau )[f]-\bA_{j,\tau}[\pi^\e_j f]\|_{H^{s-1}}\leq \|\pi_j^\e f\|_{H^{s}}+2\kappa_0 K\|  f\|_{H^{s'}},\qquad f\in H^s(\mathbb{R}).
 \end{equation*}
Moreover,  \eqref{L:FM2}  yields
  \begin{equation*} 
    2\kappa_0\|(\lambda-\bA_{j,\tau})[\pi^\e_jf]\|_{H^{s-1}}\geq 2|\lambda|\cdot\|\pi^\e_jf\|_{H^{s-1}}+ 2\|\pi^\e_j f\|_{H^s}
 \end{equation*}
 for all $-N+1\leq j\leq N$, $\tau\in[0,1],$  $\re \lambda\geq 1$, and  $f\in H^s(\mathbb{R})$.
Combining these inequalities, we conclude that
 \begin{align*}
   2\kappa_0\|\pi_j^\e(\lambda-\Phi(\tau ))[f]\|_{H^{s-1}}\geq& 2\kappa_0\|(\lambda-\bA_{j,\tau})[\pi^\e_jf]\|_{H^{s-1}}-2\kappa_0\|\pi_j^\e\Phi(\tau)[f]-\bA_{j,\tau}[\pi^\e_j f]\|_{H^{s-1}}\\[1ex]
   \geq& 2|\lambda|\cdot\|\pi^\e_j f\|_{H^{s-1}}+ \|\pi^\e_j f\|_{H^s}-2\kappa_0K\|  f\|_{H^{s'}}.
 \end{align*}
We now sum  up over $j$ to deduce from \eqref{EQNO}, Young's inequality, and the interpolation property
\begin{align}\label{IP}
[H^{s_0}(\mathbb{R}),H^{s_1}(\mathbb{R})]_\theta=H^{(1-\theta)s_0+\theta s_1}(\mathbb{R}),\qquad\theta\in(0,1),\, -\infty< s_0\leq s_1<\infty,
\end{align}
where $[\cdot,\cdot]_\theta$ denotes the complex interpolation functor, 
that there exist constants  $\kappa\geq1$  and~$\omega >1 $ such that 
  \begin{align}\label{KDED}
   \kappa\|(\lambda-\Phi(\tau ))[f]\|_{H^{s-1}}\geq |\lambda|\cdot\|f\|_{H^{s-1}}+ \| f\|_{H^s}
 \end{align}
for all   $\tau\in[0,1],$   $\re \lambda\geq \omega$, and  $f\in H^s(\mathbb{R})$.

Additionally, \eqref{L:FM1} implies that $\omega-\Phi(0) \in {\rm Isom}(H^s(\mathbb{R}), H^{s-1}(\mathbb{R}))$.
 The method of continuity \cite[Proposition I.1.1.1]{Am95} and   \eqref{KDED}  imply that also
\begin{align}\label{DEDK2}
   \omega-\Phi(1)=\omega-\p\Psi(f_0)\in {\rm Isom}(H^s(\mathbb{R}), H^{s-1}(\mathbb{R})).
 \end{align}
Combining  \eqref{KDED} (with $\tau=1$) and \eqref{DEDK2}, we conclude that~\eqref{eq:GP} holds true, cf. \cite[Chapter~I]{Am95}.
\end{proof}

 We are now in a position to prove the main result, for which we can exploit abstract theory for fully nonlinear parabolic problems from \cite[Chapter 8]{L95}.

 \begin{proof}[Proof of Theorem~\ref{MT1}]
 {\em  Well-posedness:}
  For $\alpha\in(0,1)$,  $T>0$, and a Banach space $X$, we first introduce the weighted H\"olders spaces ${\rm C}^{\alpha}_{\alpha}((0,T], X)$
 which are essential for the theory in \cite[Chapter 8]{L95}. They are defined by
 \[
 {\rm C}^{\alpha}_{\alpha}((0,T], X):=\Big\{f:(0,T]\longrightarrow X\,:\, \text{$f$ is bounded and 
 $\sup_{s\neq t}\frac{\|t^\alpha f(t)-s^\alpha f(s)\|_{X}}{|t-s|^\alpha}<\infty$}\Big\}.
 \]
 
Lemma \ref{L:REG} and  Proposition~\ref{T:GAP} show  that the assumptions of \cite[Theorem 8.1.1]{L95} 
are  satisfied for the evolution problem~\eqref{NNEP}.
This theorem ensures that for each $f_0\in H^{s}(\mathbb{R})$ there exists a positive time $T>0$ and a solution $f(\cdot;f_0)$ to~\eqref{NNEP} such that
\[ f\in {\rm C}([0,T],H^{s}(\mathbb{R}))\cap {\rm C}^1([0,T], H^{s-1}(\mathbb{R}))\cap {\rm C}^{\alpha}_{\alpha}((0,T], H^s(\mathbb{R}))\] 
for some  $\alpha\in(0,1)$.
Furthermore, it states that the solution is unique within the set
\[
  \bigcup_{\alpha\in(0,1)}{\rm C}^{\alpha}_{\alpha}((0,T],H^s(\mathbb{R})) \cap {\rm C}([0,T],H^{s}(\mathbb{R}))\cap {\rm C}^1([0,T], H^{s-1}(\mathbb{R})).
 \]
We improve this statement by showing that  the solution is actually  unique within 
$${\rm C}([0,T],H^{s}(\mathbb{R}))\cap {\rm C}^1([0,T], H^{s-1}(\mathbb{R})).$$
Indeed,  suppose $\wt f:[0,T]\longrightarrow H^s(\mathbb{R})$ is another solution to \eqref{NNEP}  satisfying the same initial condition $f_0$.  Since  \eqref{NNEP} is an  autonomous problem,
 we may assume   $f(t)\neq \wt f(t)$ for ${t\in(0,T]}$. 
 Let now $s'\in (3/2,s)$ and set $\alpha:= s-s'\in(0,1)$. 
 Using  \eqref{IP}, we find a constant~$C>0$ such that
 \begin{equation}\label{BO}
\|f(t_1)-f(t_2)\|_{H^{s'}} +\|\wt f(t_1)-\wt f(t_2)\|_{H^{s'}} \leq C|t_1-t_2|^\alpha,\qquad t_1,\, t_2\in[0, T],
 \end{equation}
which shows that $f,\, \wt f\in   {\rm C}^{\alpha}_{\alpha}((0,T], H^{s'}(\mathbb{R}))$.
 Applying the uniqueness statement from \cite[Theorem 8.1.1]{L95}  to~\eqref{NNEP} with $\Psi\in {\rm C}^{\infty}(H^{s'}(\mathbb{R}), H^{s'-1}(\mathbb{R}))$
 shows now that $f=\wt f$ on $[0,T]$.
This unique solution can be extended up to a maximal existence time $T_+(f_0)$, see \cite[Section 8.2]{L95}. 
Finally, \cite[Proposition 8.2.3]{L95} shows that the solution map defines  a semiflow on $H^s(\mathbb{R})$.
This proves~(i). \medskip

\noindent{\em  Parabolic smoothing:} The uniqueness statement in (i) enables us  to 
use a parameter trick which was successfully applied also to other problems, cf., e.g., \cite{An90, ES96, PSS15, MBV19}, in order to establish~(iia) and~(iib).
In our setting the proof details are similar to those in \cite[Theorem~1.2~(v)]{MBV18} or \cite[Theorem 1.2~(ii)]{AbMa20x} and therefore we omit them.  \medskip

\noindent{\em  Global existence:} We prove the statement by contradiction. 
Assume  there exists a maximal solution   
$f\in {\rm C}([0,T_+),H^{s}(\mathbb{R}))\cap {\rm C}^1([0,T_+), H^{s-1}(\mathbb{R}))$ to \eqref{NNEP} with $T_+<\infty$ and such that  
\begin{align}\label{BOUN1}
\sup_{[0,T_+)} \|f(t)\|_{H^s}<\infty.
\end{align}
 The bound \eqref{BOUN1} together with Lemma \ref{L:MP0}~(ii) implies that 
\begin{align}\label{BOUN2}
\sup_{[0,T_+)}\Big\| \frac{df}{dt}(t)\Big\|_{H^{s-1}}=\sup_{[0,T_+)}\|\Psi(f(t))\|_{H^{s-1}}<\infty.
\end{align}
Choosing $s'\in(3/2,s)$, we may argue as above, see \eqref{BO}, to conclude from \eqref{BOUN1} and~\eqref{BOUN2} that ${f:[0,T_+)\longrightarrow H^{s'}(\mathbb{R})}$ is uniformly continuous. 
Applying \cite[Proposition 8.2.1]{L95} to~\eqref{NNEP} with $\Psi\in {\rm C}^{\infty}(H^{s'}(\mathbb{R}), H^{s'-1}(\mathbb{R}))$, 
we may extend the solution $f$  to  an interval $[0,T_+')$ with $T_+<T_+'$ and such that 
 \[
   f\in {\rm C}([0,T_+'),H^{s'}(\mathbb{R}))\cap {\rm C}^1([0,T_+'), H^{s'-1}(\mathbb{R})).
 \]  
 The parabolic smoothing property established  in (iib) (with $s$ replaced by $s'$) implies in particular that $f\in {\rm C}^1((0,T_+'),H^{s}(\mathbb{R}))$, in contradiction to  the maximality of $f$.
 This completes our arguments.
  \end{proof}

\appendix

\section{The hydrodynamic potential near $\Gamma$ \label{appA}}

 This appendix is devoted to the study of the properties of the functions $(v^\pm,q^\pm)$ defined in~ \eqref{defvq}  near the boundary $\Gamma$.
Lemma \ref{neargamma} below establishes the corresponding part of Theorem~\ref{T:1}.

\begin{lemma}\label{neargamma}
  Given $f\in H^3(\R)$,  the functions $(v^\pm,q^\pm)$ given by \eqref{defvq} satisfy   
$$(v^\pm,q^\pm)\in   {\rm C}^1(\overline{\Omega^\pm},\mathbb{R}^2) \times   {\rm C}(\overline{\Omega^\pm}) $$ 
and solve the equations
\begin{align}\label{BCGamma}
\left.
\begin{array}{rcll}
 [v]&=&0&\mbox{on $\Gamma$,}\\[1ex]
 {}[T(v,q)](\nu\circ\Xi^{-1})&=&-\sigma(\kappa\nu)\circ\Xi^{-1}&\mbox{on $\Gamma$.}
 \end{array}
 \right\}
\end{align}
 Moreover, the velocity field $v$ on $\Gamma$ has an explicit representation in terms of nonlinear singular integral operators, given by 
\begin{equation}\label{vonGamma}
\begin{aligned}
 \hspace{-0.2cm} v_1|_\Gamma\circ\Xi&=\frac{(B_{2,2}^0(f)-B_{0,2}^0(f))[g_1-f'g_2]-B_{1,2}^0(f)[3f'g_1+g_2] -B_{3,2}^0(f)[f'g_1-g_2]}{4\pi\mu \sigma^{-1}},\\[1ex]
  \hspace{-0.2cm} v_2|_\Gamma\circ\Xi&=\frac{B_{0,2}^0(f)[f'g_1-g_2]+(B_{3,2}^0(f)-B_{1,2}^0(f))[g_1-f'g_2]-B_{2,2}^0(f)[f'g_1+3g_2]}{4\pi\mu\sigma^{-1}},
\end{aligned}
\end{equation}
with $B^0_{k,2}$ and $g_j$ defined in \eqref{BNM}, \eqref{defB0} and \eqref{derive}, \eqref{Phii}, respectively.

\end{lemma}

Before establishing  Lemma \ref{neargamma} we make the following observation.

\begin{rem}\label{R:1}
It is shown in the proof of Lemma \ref{neargamma} that  not only $v$ is continuous in $\R^2$, but  also
the first order partial derivatives of $v$, that is 
\begin{align*}
v\in {\rm C}^1(\R^2).
\end{align*}
In particular, we get
\[
 [T(v,q)](\nu\circ\Xi^{-1})=-[q]\nu\circ\Xi^{-1})\qquad\text{on $\Gamma$.}
\]
\end{rem}

\begin{proof}[Proof of Lemma \ref{neargamma}] We first recall the notation from Section \ref{fixtime}, in particular \eqref{defr}.
  From     \cite[Lemma 2.1 and Eq. (2.7)]{MBV18} we infer that, given~${\psi\in H^{1}(\mathbb{R})}$, the functions $A^\pm:\Omega^\pm\to\mathbb{R}$ defined by
\begin{align}\label{FApm}
 A^\pm (x)=\frac{1}{2\pi}\int_\mathbb{R}\frac{r^\top}{|r|^2}\psi(s)\,ds=\frac{1}{2\pi}\int_\mathbb{R}\frac{(r_1^3+r_1r_2^2, r_1^2r_2+r_2^3)^\top}{|r|^4}\psi(s)\,ds,
\end{align}
belong  to ${\rm C}(\overline{\Omega^\pm})$   and, given $\xi\in\mathbb{R}$,  we have
\be\label{jump0}
A^\pm( \xi,f(\xi))=\frac{1}{2\pi}\PV\int_\mathbb{R}\frac{(\xi-s,f(\xi)-f(s))^\top}{(\xi-s)^2+(f(\xi)-f(s))^2}\psi(s)\,ds\pm\frac{1}{2}\frac{\nu\psi}{\omega}(\xi),
\ee
where $\PV$ denotes the principal value.
Recalling \eqref{derive} and \eqref{defvq}, it follows that ${q^\pm\in {\rm C}(\overline{\Omega^\pm})}$ and 
\be\label{jumpq}
[q]\circ\Xi=\sigma\kappa.
\ee

 We next   consider the behavior of the velocity $v^\pm$ near $\Gamma$.
To this end, we first introduce the integral operators $\big[\phi\mapsto Z_n[\phi]\big]$, $n=0,\ldots,3$, by
\[ Z_n[\phi](x):=\int_\RRM\frac{r_1^{3-n}r_2^n}{|r|^4}\phi(s)\,ds,\qquad  x\in\RRM^2\setminus\Gamma,\quad \phi\in H^1(\RRM),\]
and let 
\[\{w\}^\pm(\xi):=\lim_{ \Omega^\pm\ni x\to(\xi,f(\xi))} w(x),\qquad  \xi\in \mathbb{\R},\]
be the one-sided limits of any function $w:\RRM^2\setminus\Gamma\to\RRM$ at $\Gamma$, whenever these limits exist.

Using this notation and the operators $B^0_{n,2}$ defined in \eqref{defB0}, we have from \eqref{jump0} 
\begin{equation}\label{jump1}
    \left.\begin{array}{rcl}
   \left\{Z_0[\phi]+Z_2[\phi]\right\}^\pm
   &=&B_{0,2}^0(f)[\phi]+B^0_{2,2}(f)[\phi]
   \mp\displaystyle\frac{\pi f'}{\omega^2},\\[1.5ex]
   \left\{Z_1[\phi]+Z_3[\phi]\right\}^\pm
   &=&B_{1,2}^0(f)[\phi]+B^0_{3,2}(f)[\phi]
   \pm\displaystyle\frac{\pi}{\omega^2}.
   \end{array}\right\}
\end{equation}
For $ x\in\RRM^2\setminus\Gamma$, $\phi\in H^1(\RRM)$, consider further the integral 
\be\label{defI}
I[\phi](x):=\int_\RRM \frac{(r_1r_2,r_2^2)^\top}{|r|^2}\phi'(s)\,ds.
\ee
A straightforward application of Lebesgue's dominated convergence theorem shows that $I$ extends continuously to $\Gamma$, that is $I\in {\rm C}(\R^2)$. 
Integration by parts yields
\begin{equation}\label{Irepaway}
I[\phi]=\big((Z_0-Z_2)[f'\phi]-(Z_1-Z_3),\;2Z_1[f'\phi]-2Z_2[\phi]\big)^\top\quad\mbox{
on $\RRM^2\setminus\Gamma$,}
\end{equation} and 
\begin{equation}\label{IrepGamma}
I[\phi]\circ\Xi=\big((B^0_{0,2}(f)-B^0_{2,2}(f))[f'\phi] -(B^0_{1,2}(f)-B^0_{3,2}(f))[\phi],\,2B^0_{1,2}(f)[f'\phi]-2B_{2,2}(f)[\phi]\big)^\top\!\!.
\end{equation}
So, by the continuity of $I$,
\begin{equation}\label{jump2}
\left.
\hspace{-0,55cm}\begin{aligned}
 \left\{(Z_0-Z_2)[f'\phi]-(Z_1-Z_3)\right\}^\pm &= (B^0_{0,2}(f)-B^0_{2,2}(f))[f'\phi] -(B^0_{1,2}(f)-B^0_{3,2}(f))[\phi],\\
    \left\{2Z_1[f'\phi]-2Z_2[\phi]\right\}^\pm &=2B^0_{1,2}(f)[f'\phi]-2B_{2,2}(f)[\phi].
   \end{aligned}\right\}
   \end{equation}
Observe that Eqns. \eqref{jump1}, \eqref{jump2} constitute a linear system of the form
\begin{equation}\label{lsys}
\left\{\sum_{n=0}^3 Z_n[a_{in}\phi]\right\}^\pm=\sum_{n=0}^3B^0_{n,2}(f)[a_{in}\phi]\pm J_i\phi,\qquad i=1,\ldots,4,\quad \phi\in H^1(\RRM),
\end{equation}
with coefficient matrix
\[\cla:=(a_{in}):=\left(\begin{array}{ccccccc}
1&&0&&1&&0\\
0&&1&&0&&1\\
f'&&-1&&-f'&&1\\
0&&2f'&&-2&&0
\end{array}\right)\]
and with $(J_1,\ldots,J_4):=\displaystyle\frac{\pi}{\omega^2}(-f',\,1,\,0,\,0)^\top$.
The matrix $\cla$ is regular and has inverse
\[\cla^{-1}:=(a^{ki}):=\frac{1}{2\omega^2}
\left(\begin{array}{ccccccc}
2+{f'}^2 && -f' && -f' && 1\\
f' && 1 && -1 && f'\\
{f'}^2 && f' &&  -f' && -1\\
-f' && 1+2{f'}^2 && 1 && -f'
\end{array}\right).\]
Observe that for $k=0,\ldots,3$, $i=1,\ldots,4$,  and $\phi\in H^1(\RRM)$ we have $a^{ki}\phi\in H^1(\RRM)$.

For $k=0,\ldots,3$, replace $\phi$ by $a^{ki}\phi$ in the $i$-th equation of \eqref{lsys} and sum over $i$. This gives
\[\{Z_k[\phi]\}^\pm=B^0_{k,2}(f)[\phi]\pm \sum_{i=1}^4 a^{ki}J_i\phi,\]
or equivalently
\begin{equation*}
\begin{aligned}
 \frac{1}{2\pi}\left\{\int_\R {\frac{(r_1^3,r_1^2 r_2,r_1r_2^2,r_2^3)^\top}{|r|^4}} \phi\, ds\right\}^\pm& =\frac{(B_{0,2}^0(f),B_{1,2}^0(f),B_{2,2}^0(f),B_{0,2}^0(f))^\top[\phi]}{2\pi}\\[1ex]
&\hspace{0,45cm}\mp\Big(\frac{f'^3+3f'}{4\omega^4},\frac{f'^2-1}{4\omega^4},\frac{f'^3-f'}{4\omega^4},-\frac{3f'^2+1}{4\omega^4}\Big)^\top\phi.
 \end{aligned}
 \end{equation*}
In view of this, we get directly from \eqref{stokesdiff} and \eqref{defvq}$_1$ that $[v]=0$  and the representation~\eqref{vonGamma} is valid. 

 Moreover, first differentiating with respect to $x$  under the integral in \eqref{defvq}$_1$ and then using integration by parts and \eqref{derive} we find
\[\partial_iv(x)=  \sigma \int_\mathbb{R} \partial_i \clu^1(r)(f'\kappa)(s)-\partial_i \clu^2(r)\kappa(s)\,ds,\quad i=1,\,2.\]
It is now straightforward to check that $v^\pm\in {\rm C}^1(\overline{\Omega^\pm})$ and $[\nabla v]=0$. 
Together with \eqref{jumpq}, this implies \eqref{BCGamma}$_2$, and the proof is complete. 
\end{proof}

\section{The hydrodynamic potential in the far-field limit\label{appB}}

 In this appendix we prove that the functions $(v,q)$ defined in \eqref{defvq} satisfy the far-field boundary condition~\eqref{fixtimeeq}$_5$.
While the claim for $q$ follows directly from \cite[Lemma 2.1]{MBV18}, proving that the velocity vanishes 
at infinity is more elaborate and necessitates some preparation. 

Recall  that $f\in H^3(\R)$ is fixed.
  Using the notation \eqref{defr} once again, we  define functions~${\big[\phi\mapsto(F,G)[\phi]\big]}$ according to (cf.~\eqref{defI})
\begin{align}\label{BB121}
(F,G)^{\top}[\phi](x):=I[\phi](x)=\int_\mathbb{R}\frac{(r_1r_2,r_2^2)^\top}{|r|^2}\phi'(s)\,ds,\qquad\phi\in H^{1}(\mathbb{R}),  x\in\RRM^2.
\end{align}
We recall from Appendix \ref{appA} that  $ F,\, G \in{\rm C}(\R^2)$.
Moreover, from \eqref{IrepGamma} and Lemma \ref{L:MP0}~(ii), we get
\be\label{hoelderint}
(F,G)|_\Gamma\circ\Xi\in {\rm C}^\beta(\mathbb{R})
\ee
for some $\beta\in(0,1)$, and 
\be\label{intlim}
(F,G)(\xi,f(\xi))\to 0\qquad\mbox{for $|\xi|\to\infty$.}
\ee

We first prove a bound for $F$ and $G$ at moderate distances from the interface, in terms of their values at the interface.

\begin{lemma}[Vertical differences]\label{vertdiff} 
 Given $f\in H^3(\R)$ and $\phi\in H^2(\R)$, there exist constants $\alpha\in(0,1)$ and $C_0>0$ such that for 
\[ x\in S:=\mathbb{R}\times[-\|f\|_\infty-1,\|f\|_\infty+1]\] we have
\[|(F,G)( x)-(F,G)(x_1,f(x_1))|\leq C_0|x_2-f(x_1)|^\alpha.\]
\end{lemma}
\begin{proof}
We    show the estimate for $F$ only, the arguments for $G$ are analogous with some obvious modifications.
Given $x\in S,$ we  choose $ \bar x:=(\xi_0,f(\xi_0))$ such that
\[ |x-\bar x|=\min\{|x-(\xi,f(\xi))|\,|\,\xi\in\mathbb{R}\}.\]
After a change of variables we split
\begin{align*}
&\hspace{-0,5cm}|F(x)-F( \xi_0,f(\xi_0) )|\\[1ex]
&\leq \int_\mathbb{R}|\phi'(x_1-s)-\phi'(\xi_0-s)|
\frac{|s||x_2-f(x_1-s)|}{s^2+(x_2-f(x_1-s))^2}\,ds\\[1ex]
&\hspace{0,5cm}+\int_\mathbb{R}|\phi'(\xi_0-s)|\left|\frac{s(x_2-f(x_1-s))}{s^2+(x_2-f(x_1-s))^2}-
\frac{s(f(\xi_0)-f(\xi_0-s))}{s^2+(f(\xi_0)-f(\xi_0-s))^2}\right|\,ds.
\end{align*}
We estimate the terms on the right separately. 
For the first one we use $\phi'\in {\rm C}^{2\alpha}(\mathbb{R})$ for some $\alpha\in (0,\min\{\beta,1/2\}]$ to obtain
\begin{align*}
&\hspace{-0,5cm}\int_\mathbb{R}|\phi'(x_1-s)-\phi'(\xi_0-s)|\frac{|s||x_2-f(x_1-s)|}{s^2+(x_2-f(x_1-s))^2}\,ds\\[1ex]
&\leq C|x_1- \xi_0|^\alpha\int_\mathbb{R}|\phi'(x_1-s)-\phi'(\xi_0-s)|^{1/2}
\frac{|s||x_2-f(x_1-s)|}{s^2+(x_2-f(x_1-s))^2}\,ds\\[1ex]
&\leq C|x_1-\xi_0|^\alpha\|\phi'\|_2^{1/2}\left(
\int_\mathbb{R}\left(\frac{|s||x_2-f(x_1-s)|}{s^2+(x_2-f(x_1-s))^2}\right)^{4/3}\,ds\right)^{3/4}\\[1ex]
&\leq C|x_1-\xi_0|^\alpha\left(\int_\mathbb{R}\min\{1,s^{-4/3}\}\,ds\right)^{3/4}
\leq C | x-\overline x|^\alpha.
\end{align*}
To estimate the second term we write for brevity $\zeta:=x_2-f(x_1-s)$, $\zeta_0:=f(\xi_0)-f(\xi_0-s)$ and observe
\begin{align*}
    \left|\frac{s\zeta}{s^2+\zeta^2}-\frac{s\zeta_0}{s^2+\zeta_0^2}\right|
    &=\left|\frac{s^3-s\zeta\zeta_0}{(s^2+\zeta^2)(s^2+\zeta_0^2)}\right||\zeta-\zeta_0|\leq \Big(\frac{|s|}{s^2+\zeta^2}+\frac{|\zeta|}{s^2+\zeta^2}\Big)|\zeta-\zeta_0|\\[1ex]
    &\leq \frac{2}{(s^2+\zeta^2)^{1/2}}(| x_2-f(\xi_0)|+C|x_1-\xi_0|)\leq \frac{C}{(s^2+\zeta^2)^{1/2}}|x-\overline x| ,
\end{align*}
to obtain
\[\int_\mathbb{R}|\phi'(\xi_0-s)|\left|\frac{s\zeta}{s^2+\zeta^2}-\frac{s\zeta_0}{s^2+\zeta_0^2}\right|\,ds
\leq C\int_\mathbb{R}|\phi'(\xi_0-s)|\frac{| x-\overline x|}{(s^2+\zeta^2)^{1/2}}\,ds.\]
We split the integral on the right. For $|s|<1$ we use the minimality property of $\bar x$ to obtain
\[\frac{| x-\bar x|}{(s^2+\zeta^2)^{1/2}}=\frac{| x-\bar x|^{1/2}}{(s^2+\zeta^2)^{1/4}}\frac{|x-\bar x|^{1/2}}{| x-(x_1-s,f(x_1-s))|^{1/2}}\leq  \frac{| x-\bar x|^{1/2}}{s^{1/2}}\]
and thus
\[\int_{\{|s|<1\}}|\phi'(\xi_0-s)|\frac{| x-\bar x|}{(s^2+\zeta^2)^{1/2}}\,ds
\leq C\|\phi'\|_\infty|x-\bar x|^{1/2}\int_{\{|s|<1\}}s^{-1/2}\,ds
\leq C|x-\bar x|^{1/2}.\]
For $|s|>1$ we estimate directly 
\[\int_{\{|s|>1\}}|\phi'(\xi_0-s)|\frac{| x-\bar x|}{(s^2+\zeta^2)^{1/2}}\,ds
\leq | x-\bar x|\|\phi'\|_2\Big(\int_{\{|s|>1\}}s^{-2}\,ds\Big)^{1/2}\leq C| x-\bar x|.\]
Summarizing and using the boundedness of $F$ on $S$ (which follows by applying H\"older's inequality to \eqref{BB121}) we get
\[|F(x)-F(\bar x)|\leq C| x-\bar x|^\alpha,\]
and consequently, using $\eqref{hoelderint}$,
\begin{align*}
    |F( x)-F(x_1,f(x_1))|&\leq |F(x)-F(\xi_0,f( \xi_0))|+|F( \xi_0,f( \xi_0))-F(x_1,f(x_1))|\\[1ex]
    &\leq C(| x-\bar x|^\alpha+|x_1- \xi_0|^\alpha)\leq C | x-\bar x|^\alpha\\[1ex]
    &\leq C|x_2-f(x_1)|^\alpha,
\end{align*}
where the minimality property of $\bar x$ has been used again in the last step.\end{proof}

 We next prove that the functions $F,\,G$ defined in \eqref{BB121} vanish at infinity.

\begin{lemma}\label{asyFG}
 Given $f\in H^3(\R)$ and $\phi\in H^{2}(\R),$ we have
\[(F,G)(x)\to 0\qquad\mbox{for $|x|\to\infty$.}\]
\end{lemma}
\begin{proof}
We will show the result for $F$, the proof for $G$ is essentially analogous.
The result is proved in the following three steps:
\begin{itemize}
    \item[(i)] For any $\eps>0$ there are $ \xi_0>0$, $\delta\in(0,1)$ such that $|x_1|> \xi_0$ and $|x_2|<\delta$  imply~$|F(x)|<\eps$,
    \item[(ii)] For any $\eps>0$ there is a $\delta\in(0,1)$  such that $| x_2|>\delta^{-1}$ implies $|F(x)|<\eps$ for all~${x_1\in\RRM}$,
    \item[(iii)] For all $\eps>0$, $\delta\in(0,1)$ there is a $\xi_1>0$ such that $|x_1|>\xi_1$ and
    $\delta\leq|x_2|\leq \delta^{-1}$ imply $|F(x)|<\eps$.
 \end{itemize}

To show (i), fix $\eps>0$, choose first $\delta>0$ small enough to ensure $C_0(2\delta)^\alpha<\eps/2$ with~$C_0$ and~$\alpha$ from Lemma \ref{vertdiff}, 
and then $\xi_0$ large enough to guarantee that $|f(x_1)|<\delta$ and $|F(x_1,f(x_1))|<\eps/2$ whenever $|x_1|> \xi_0$, which is possible by \eqref{intlim}. 
Now it follows from Lemma \ref{vertdiff} that $|F(x)|<\eps$  whenever $|x_1|> \xi_0$, $|x_2|<\delta$. 
 
For (ii) we have to prove that $ F( x)\to0$ for $|x_2|\to\infty$, uniformly in $x_1\in\mathbb{R}$. 
From \eqref{Irepaway} we immediately have
 \[|F(x)|\leq C \int_\mathbb{R}\frac{|\phi(s)|(|r_1|+|r_2|)}{|r|^2}\,ds.\]
 Choosing $|x_2|>2\|f\|_\infty,$ we get $|x_2|/2\leq|r_2|\leq 3|x_2|/2$ and using the Cauchy-Schwarz inequality we get
\begin{align*}
    |F(x)|&\leq C\int_\mathbb{R}\frac{|\phi(s)|(|r_1|+|x_2|)}{4r_1^2+|x_2|^2}\,ds
    \leq C\|\phi\|_2\Big(\int_\mathbb{R}\Big(\frac{|r_1|+|x_2|}{4r_1^2+|x_2|^2}\Big)^2\,ds\Big)^{1/2}\\[1ex]
    &\leq\frac{C}{|x_2|^{1/2}}\Big(\int_\mathbb{R}\frac{(t+1)^2}{(4t^2+1)^2}\,dt\Big)^{1/2}\leq\frac{C}{|x_2|^{1/2}},
\end{align*}
where we changed variables according to $t:=(x_1-s)/|x_2|$ in the last step.
This proves (ii).

To show (iii), let $\eps>0$ and $\delta\in(0,1)$ be fixed. 
Given $x_2\in\R$ with $\delta\leq|x_2|\leq \delta^{-1}$, it follows that
  $| r_2|\leq\delta^{-1}+\|f\|_\infty$.
Choose $s_0>0$ large enough to ensure 
\begin{align}
\bullet \quad&(\delta^{-1}+\|f\|_\infty)\sqrt{2\pi/\delta}\left(\int_{|s|>s_0} {\phi'}^2(s)\,ds\right)^{1/2}<\eps/2\quad\text{and}\label{chs_0}\\[1ex]
\bullet \quad&|f(s)|<\delta/2 \quad\mbox{for $|s|>s_0$.}\label{chs_1}
\end{align}
Given  $x_1\in\R$ with $ |x_1|>2s_0$,  in view of \eqref{BB121} we  obtain   the estimate
\[|F( x)|\leq \int_\mathbb{R}\frac{|r_1r_2|}{ |r|^2}|{\phi'}(s)|\,ds.\]
We split the integral on the right as follows.
If $|s|<s_0,$ then $|r_1|>|x_ 1|/2$  and 
\[\int_{\{|s|<s_0\}}\frac{|r_1r_2|}{|r|^2}|{\phi'}(s)|\,ds\leq \frac{4s_0(\delta^{-1}+\|f\|_\infty)\|{\phi'}\|_\infty}{|x_1|}=\frac{C}{|x_1|}<\frac{\eps}{2},\]
provided that $|x_1|>\xi_1$  with $\xi_1> 2s_0$ chosen sufficiently large. 

If $|s|>s_0$, then $|r_2|>\delta/2$ by \eqref{chs_1}  and, using \eqref{chs_0} and  the Cauchy-Schwarz inequality, we get
\begin{align*}
    \int_{\{|s|>s_0\}}\frac{|r_1r_2|}{|r|^2}|{\phi'}(s)|\,ds&\leq(\delta^{-1}+\|f\|_\infty)\int_{\{|s|>s_0\}}\frac{| r_1|}{ r_1^2+(\delta/2)^2}|{\phi'}(s)|\,ds\\[1ex]
&\leq (\delta^{-1}+\|f\|_\infty)\left(\int_\mathbb{R}\frac{1}{t^2+(\delta/2)^2 }\,dt\right)^{1/2}
\left(\int_{|s|>s_0} {\phi'}^2(s)\,ds\right)^{1/2}\\[1ex]
&\leq (\delta^{-1}+\|f\|_\infty)\sqrt{2\pi/\delta}
\left(\int_{|s|>s_0} {\phi'}^2(s)\,ds\right)^{1/2}<\frac{\eps}{2}.
\end{align*}
This completes the proof.
\end{proof}

 We are now in a position to prove the desired  decay behavior.

\begin{lemma}\label{L:BB3}
Given $f\in H^3(\R)$, the functions $(v^\pm,q^\pm)$ given by \eqref{defvq}  satisfy
\[(v^\pm,q^\pm)(x) \to 0\qquad \mbox{for $|x|\to\infty$.}\]
\end{lemma}
\begin{proof}
Since $f\in H^3(\R)$, for $\psi\in H^{1}(\mathbb{R})$  the functions $A_\pm$ defined in \eqref{FApm} satisfy $A^\pm(x)\to0 $ for $|x|\to\infty$, cf. \cite[Lemma 2.1]{MBV18}.
Recalling the definition \eqref{defvq}$_2$ of $q$, it immediately follows (since $g_k'\in H^1(\R)$, $k=1,\, 2$) that 
\[q(x)\to 0 \qquad\mbox{for $|x|\to\infty$}.\]
Moreover, the equation \eqref{defvq}$_1$ together with the relation $\p_s(r_1^2/|r|^2)=-\p_s( r_2^2/ |r|^2)$ enable us to write 
\[
v(x)=\frac{\sigma}{4\pi\mu}
\int_\mathbb{R}\frac{1}{|r|^2}
\left(
\begin{array}{ccc}
-r_2^2&& r_1r_2\\
 r_1r_2&& r_2^2
\end{array}\right)g'(s)\,ds
-\frac{\sigma}{4\pi\mu}\int_\mathbb{R}\frac{r_1+f'(s)r_2}{|r|^2}g(s)\,ds,\quad x\in\R^2\setminus\Gamma.\]
Since $g_k\in H^2(\R)$, $k=1,\,2$, Lemma~\ref{asyFG} implies that the first integral vanishes at infinity.
Due to \cite[Lemma 2.1]{MBV18}, also the second integral vanishes in the far field limit, hence
\[v^\pm(x)\to 0 \qquad\mbox{for $|x|\to\infty$}.\]
\end{proof}

\section{Smoothness of some nonlinear operators\label{appsmooth}}
 In this appendix we establish the smoothness of certain nonlinear operators  we are confronted with in Section~\ref{Sec:3}.
In the following $r\in(1/2,1)$ is fixed.

\begin{lemma}\label{mult}
Let $\psi:\mathbb{R}\longrightarrow\mathbb{R}$ be locally Lipschitz continuous, i.e.
for any $K>0$ there is a constant $L_K>0$ such that
\[|\psi(\xi)-\psi(\xi')|\leq L_K|\xi-\xi'|\quad\mbox{ for all $\xi,\,\xi'\in[-K,K]$}.\]
Let $z\in H^r(\mathbb{R})$. 
Then, the map
$[h\mapsto(\psi\circ z)h]$
is in $\kL(H^r(\mathbb{R}))$.
\end{lemma}
\begin{cor}\label{isom}
If, in addition to the assumptions of Lemma {\rm\ref{mult}}, $\psi$ is strictly positive, then the linear operator $[h\mapsto(\psi\circ z)h]$ is an isomorphism of $H^r(\mathbb{R})$.
\end{cor}
\begin{proof}[Proof of Lemma~\ref{mult}]
Recall that a norm in $H^r(\mathbb{R})$,  which is equivalent to the standard norm, is given by 
\[\|z\|^2_{H^r}=\|z\|_2^2+[z]_{H^r}^2 \quad\text{with}\quad [z]^2_{H^r}=
\int_\mathbb{R}\frac{\|\tau_\eta z-z\|_2^2}{| \eta|^{1+2r}}\,d\eta,\]
where $\tau_\eta:=[z\mapsto z(\cdot-\eta)]$ denotes the right shift operator. 
Let $K=\|z\|_\infty$ and observe
\begin{align*}
    \|\psi\circ z\|_\infty&\leq |\psi(0)|+KL_K,\\[1ex]
    \|(\psi\circ z)h\|_2&\leq  \|\psi\circ z\|_\infty\|h\|_2\leq C \|h\|_2,\\[1ex]
    [(\psi\circ z)h]_{H^r}^2&=\int_\mathbb{R}\frac{\|\tau_\eta((\psi\circ z)h)-(\psi\circ z)h\|^2_2}{|\eta|^{1+2r}}\,d\eta,\\[1ex]
                             \|\tau_\eta((\psi\circ z)h)-(\psi\circ z)h\|_2
    &\leq \|\psi\circ z\|_\infty\|\tau_\eta h-h\|_2+\|h\|_\infty L_K\|\tau_\eta z-z\|_2
\end{align*}
to find
\[\|(\psi\circ z)h\|_{H^r}\leq C\|h\|_{H^r},\quad h\in H^r(\R).\]
\end{proof}

 We now use  Corollary~\ref{isom} to establish in Lemma \ref{phismooth} the smoothness of three nonlinear operators.
Lemma~\ref{phismooth} is the main argument in the proof of Lemma~\ref{L:MP1}.

\begin{lemma}\label{phismooth}
The maps
\[\mbox{\rm (i)}\quad
z\mapsto\sqrt{z^2+1}-1,\qquad
\mbox{\rm (ii)}\quad
z\mapsto\frac{z^2}{\sqrt{z^2+1}+z^2+1},\qquad
\mbox{\rm (iii)}\quad
z\mapsto\frac{z}{\sqrt{z^2+1}}
\]
are (${\rm C}^\infty$-)smooth from $H^r(\mathbb{R})$ to $H^r(\mathbb{R})$.
\end{lemma}
\begin{proof}
Fix any $z_0\in H^r(\mathbb{R})$. Note first that
\[u_0:=\sqrt{z_0^2+1}-1=\frac{z_0}{\sqrt{z_0^2+1}+1}z_0\in H^r(\mathbb{R})\]
by Lemma \ref{mult}. 
Furthermore, a pair $(u,z)\in\left(H^r(\mathbb{R})\right)^2$ close to $(u_0,z_0)$ is a solution to the equation
\[F(z,u)=u^2+2u-z^2=0\]
if and only if $u=U(z):=\sqrt{z^2+1}-1$. 
The mapping $F$ is smooth from $\left(H^r(\mathbb{R})\right)^2$ to $H^r(\mathbb{R})$, and its Fr\'echet derivative $\partial_uF(z_0,u_0)$ is given by 
\[\partial_uF(z_0,u_0)h=2h\sqrt{z_0^2+1} ,\]
which is an isomorphism of $H^r(\mathbb{R})$ by Corollary \ref{isom}.
Now (i) follows from (the ${\rm C}^\infty$-version of) the Implicit Function theorem. 

The proof of (ii) is similar. First we note that
\[v_0:=\frac{z_0^2}{\sqrt{z_0^2+1}+z_0^2+1}\in H^r(\mathbb{R})\]
by Lemma \ref{mult}, and then we apply the Implicit Function theorem near $(z_0,v_0)$ to the equation
\[G(z,v)=z^2v+2v+(\sqrt{z^2+1}-1)v-z^2=0,\]
where $G$ is smooth from $\left(H^r(\mathbb{R})\right)^2$ to $H^r(\mathbb{R})$
due to (i).
Furthermore
\[\partial_vG(z_0,v_0)=[h\mapsto(\sqrt{z_0^2+1}+z_0^2+1)h]\]
is an isomorphism due to Corollary \ref{isom}. 

Finally, (iii) follows from (ii) and the identity
\[\frac{z}{\sqrt{z^2+1}}=z-z\frac{z^2}{\sqrt{z^2+1}+z^2+1}.\]
\end{proof}

 In the final part of this appendix we establish the smoothness of certain multilinear singular operators $B^k_{n,m}$.
Corollary~\ref{Bsmooth} below implies in particular that, given $n,\, m\in\N$, the mapping 
$$[f\mapsto B^0_{n,m}(f)[\cdot]]:H^s(\R)\to \kL( H^{s-1}(\mathbb{R})),$$ 
where $s\in(3/2,2)$ and with $B^0_{n,m}$ as defined in \eqref{defB0}, is smooth.
This property is essential when proving Lemma \ref{L:REG}.

We start by  introducing, for given  $k,\,n,\, m\in\N,$ operators
\[B_{n,m}^k:\;H^s(\mathbb{R})\longrightarrow\kL^k_{\rm sym}
(H^s(\mathbb{R}),\kL(H^{s-1}(\mathbb{R})))\]
by
\[B^k_{n,m}(f)[f_1,\ldots,f_k][h]=B_{n+k,m}(f,\ldots,f)
[f,\ldots,f,f_1,\ldots,f_k,h],\quad h\in H^s(\mathbb{R}),\]
where $B_{n+k,m}$ are  defined in \eqref{BNM}.
Observe that for $k=0$ this definition agrees with~\eqref{defB0}.
\begin{lemma}\label{Bdiff}
Given $k,\, n,\, m\in \N$ and $f\in H^s(\R)$,  the map $B_{n,m}^k$ is Fr\'echet differentiable at~$f$, and the Fr\'echet derivative $\partial B^k_{n,m}(f)$ is given by
\[\partial B^k_{n,m}(f)[g][f_1,\ldots,f_k]=nB^{k+1}_{n-1,m}(f)[f_1,\ldots,f_k,g]-2mB^{k+1}_{n+1,m+1}(f)[f_1,\ldots,f_k,g]\]
for $g\in H^s(\R)$.
\end{lemma}

 A straightforward consequence of Lemma~\ref{Bdiff} is the following corollary.

\begin{cor}\label{Bsmooth}
Given $k,\, n,\, m\in \N$, we have 
$$ B_{n,m}^k\in {\rm C}^\infty(H^s(\mathbb{R}),\kL^k_{\rm sym}
(H^s(\mathbb{R}),\kL( H^{s-1}(\mathbb{R}))).$$
\end{cor}

\begin{proof}[Proof of Lemma~{\rm\ref{Bdiff}}]
Given $f,\, g,\, f_1,\ldots,\, f_k\in H^s(\mathbb{R}),$ we consider the remainder 
\begin{align*}
    R(f,g)[f_1,\dots,f_k]
     &:=\big(B_{n,m}^k(f+g)-B^k_{n,m}(f)\big)[\fks]\\
     &\hspace{0,55cm}-nB^{k+1}_{n-1,m}(f)[\fks,g]+2m B^{k+1}_{n+1,m+1}(f)[\fks,g]
\end{align*}
and represent it (by elementary algebraic operations)
as
\begin{align*}
    &\hspace{-0.5cm}R(f,g)[\fks]\\
    =&\sum_{j=0}^{n-2}(n-j-1)B_{n+k,m}(\fgs)[\underbrace{\fgs}_{j},\underbrace{\fs}_{n-j-2},\fks,g,g,\cdot]\\
    &-\sum_{l=0}^{m-1}B_{n+k+2,m+1}(\underbrace{\fgs}_{m-l},\underbrace{\fs}_{l+1})
    [\fs,n(2f+g)+f,\fks,g,g,\cdot]\\
&+2\sum_{l=0}^{m-1}\sum_{j=0}^{m-l-1}
B_{n+k+4,m+2}(\underbrace{\fgs}_{m-l-j}
\underbrace{\fs}_{l+j+2})
[\fs,2f+g,\fks,g,g,\cdot].
\end{align*}
Sums with negative upper summation limit are to be neglected.
For $\|g\|_{H^s}\leq 1$, we obtain from Lemma~\ref{L:MP0}~(ii) 
\[\|R(f,g)[\fks]\|_{\kL(H^{s-1}(\mathbb{R}))}
\leq C\|g\|_{H^s}^2\prod_{i=1}^k\|f_i\|_{H^s}.\]
This implies the result.
\end{proof}

 \section{Some auxiliary results related to localization}\label{Sec:B}

The following commutator estimate is used in the proof  of  Lemmas~\ref{L:B4} and~\ref{L:B5} below.

\begin{lemma}\label{L:B0} 
Let $n,\, m \in \N$,   $s\in(3/2, 2)$, $f\in H^s(\mathbb{R})$, and  ${\varphi\in {\rm C}^1(\mathbb{R})}$ with uniformly continuous derivative $\varphi'$ be given. 
Then, there exist  a constant $K$ that depends only on $ n, $ $m, $ $\|\varphi'\|_\infty, $ and $\|f\|_{H^s}$  such that 
 \begin{equation}\label{LB2}
  \|\varphi B_{n,m}^0(f)[h]- B_{n,m}^0(f)[\varphi h]\|_{H^{1}}\leq K\| h\|_{2}
 \end{equation}
for all   $h\in L_2(\mathbb{R})$.
\end{lemma}
\begin{proof}
See \cite[Lemma~4.4]{AbMa20x}.
\end{proof}
 
 The next four lemmas  are used in the proof of Proposition~\ref{T:AP}. We recall the definition of an $\eps$-localization family from Section \ref{Sec:3}.

\begin{lemma}\label{L:B1} 
Let $n,\, m \in \N$, $3/2<s'<s<2$, and  $\nu\in(0,\infty)$ be given. 
Let further  $f\in H^s(\mathbb{R})$ and  $\oo\in \{1\}\cup H^{s-1}(\mathbb{R})$.
For any sufficiently small $\e\in(0,1)$, there is
a constant $K=K(\e, n, m, \|f\|_{H^s},\|\oo\|_{H^{s-1}})$  such that 
 \begin{equation*} 
  \Big\|\pi_j^\e\oo B_{n,m}^0(f)[ h]-\frac{\oo(\xi_j^\e)(f'(\xi_j^\e))^n}{[1+(f'(\xi_j^\e))^2]^m}B_{0,0}[\pi_j^\e h]\Big\|_{H^{s-1}}\leq \nu \|\pi_j^\e h\|_{H^{s-1}}+K\| h\|_{H^{s'-1}} 
 \end{equation*}
for all $|j|\leq N-1$ and  $h\in H^{s-1}(\mathbb{R})$.
\end{lemma}  
\begin{proof}
See \cite[Lemma~4.5]{AbMa20x}.
\end{proof}

The next two lemmas are the analogues of Lemma \ref{L:B1}  corresponding to the case $j=N$.

\begin{lemma}\label{L:B2} 
Let $n,\, m \in \N$,  $3/2<s'<s<2$, and  $\nu\in(0,\infty)$ be given. 
Let further  $f\in H^s(\mathbb{R})$ and  $\oo\in  H^{s-1}(\mathbb{R})$.
For any sufficiently small  $\e\in(0,1)$, there is
a constant $K=K(\e, n, m, \|f\|_{H^s},\|\oo\|_{H^{s-1}})$  such that 
  \begin{equation*}
  \|\pi_N^\e\oo B_{n,m}^0(f)[h]\|_{H^{s-1}}\leq \nu \|\pi_N^\e h\|_{H^{s-1}}+K\| h\|_{H^{s'-1}}
 \end{equation*} 
 for $h\in H^{s-1}(\mathbb{R})$.
\end{lemma}  
\begin{proof}
See \cite[Lemma~4.6]{AbMa20x}.
\end{proof}

Lemma \ref{L:B3} is the counterpart of Lemma \ref{L:B2} in the case when $\oo=1$.

\begin{lemma}\label{L:B3} 
Let $n,\, m \in \N$,  $3/2<s'<s<2$, and  $\nu\in(0,\infty)$ be given. 
Let further  ${f\in H^s(\mathbb{R})}$.
For any sufficiently small  $\e\in(0,1)$, there is
a constant $K=K(\e, n, m, \|f\|_{H^s})$  such that 
 \begin{equation*} 
  \|\pi_N^\e B_{0,m}^0(f)[ h]-B_{0,0}[\pi_N^\e h]\|_{H^{s-1}}\leq \nu \|\pi_N^\e h\|_{H^{s-1}}+K\| h\|_{H^{s'-1}}
 \end{equation*}
 and 
  \begin{equation*}
  \|\pi_N^\e B_{n,m}^0(f)[h]\|_{H^{s-1}}\leq \nu \|\pi_N^\e h\|_{H^{s-1}}+K\| h\|_{H^{s'-1}},\qquad n\geq 1,
 \end{equation*}
 for  all $h\in H^{s-1}(\mathbb{R})$.
\end{lemma}
\begin{proof}
See \cite[Lemma~4.7]{AbMa20x}.
\end{proof}

We now prove a lemma which deals with a similar situation as in  Lemma~\ref{L:B1}.

\begin{lemma}\label{L:B4} 
Let $n,\, m \in \N$, $3/2<s'<s<2$, and  $\nu\in(0,\infty)$ be given. 
Let further  $f\in H^s(\mathbb{R})$, $a\in  H^{s-1}(\mathbb{R})$, and $\oo\in \{1\}\cup H^{s-1}(\mathbb{R})$.
For any sufficiently small $\e\in(0,1)$, there is
a constant $K$ depending on $\e,$ $ n,$ $ m,$ $ \|f\|_{H^s},$  $\|a\|_{H^{s-1}},$  and $ \|\oo\|_{H^{s-1}}$ (if $\oo\neq1$)  such that 
 \begin{equation*} 
\Big\|\pi_j^\e\oo B_{n,m}^0(f)[ ah]-\frac{a(\xi_j^\e)\oo(\xi_j^\e)(f'(\xi_j^\e))^n}{[1+(f'(\xi_j^\e))^2]^m}B_{0,0}[\pi_j^\e h]\Big\|_{H^{s-1}}\leq \nu \|\pi_j^\e h\|_{H^{s-1}}+K\| h\|_{H^{s'-1}}
 \end{equation*}
for all $|j|\leq N-1$ and  $h\in H^{s-1}(\mathbb{R})$.
\end{lemma}  
\begin{proof}
Let first $\oo\in H^{s-1}(\mathbb{R})$ (the case $\oo=1$ is similar).
It is suitable to decompose 
\begin{align*}
 \pi_j^\e\oo B_{n,m}^0(f)[ ah]-\frac{a(\xi_j^\e)\oo(\xi_j^\e)(f'(\xi_j^\e))^n}{[1+(f'(\xi_j^\e))^2]^m}B_{0,0}[\pi_j^\e h]=\oo(T_1+T_2)+a(\xi_j^\e)T_3, 
\end{align*}
where
\begin{align*}
T_1&:=\pi_j^\e B_{n,m}^0(f)[(a-a(\xi_j^\e))h]- B_{n,m}^0(f)[\pi_j^\e(a-a(\xi_j^\e))h],\\[1ex]
T_2&:= B_{n,m}^0(f)[ \pi_j^\e(a-a(\xi_j^\e))h],\\[1ex]
T_3&:=\pi_j^\e\oo B_{n,m}^0(f)[ h]-\frac{ \oo(\xi_j^\e)(f'(\xi_j^\e))^n}{[1+(f'(\xi_j^\e))^2]^m}B_{0,0}[\pi_j^\e h].
\end{align*} 
Lemma~\ref{L:B0} yields
\[
\|\oo T_1\|_{H^{s-1}}\leq K\|(a-a(\xi_j^\e))h\|_2\leq K\|h\|_2.
\] 
Besides, using Lemma~\ref{L:MP0}~(ii),  \eqref{MES}, the identity $\chi_j^\e\pi_j^\e=\pi_j^\e$, and the fact that~${a\in {\rm C}^{s-3/2}(\mathbb{R})}$,  we see that
\begin{align*}
\|\oo T_2\|_{H^{s-1}}&\leq C\|\pi_j^\e(a-a(\xi_j^\e))h\|_{H^{s-1}} \leq C\|\chi_j^\e(a-a(\xi_j^\e))\|_\infty\|\pi_j^\e h\|_{H^{s-1}}+K\| h\|_{H^{s'-1}}\\[1ex]
&\leq \frac{\nu}{2} \|\pi_j^\e h\|_{H^{s-1}}+K\| h\|_{H^{s'-1}}
\end{align*}
 provided that $\e$ is sufficiently small (where $C=C(n,m,\|\oo\|_{H^{s-1}}, \|f\|_{H^s})$).
 Since $T_3$ can be estimated by using Lemma~\ref{L:B1}, we have established the desired claim.
\end{proof}

 Lemma~\ref{L:B5} below can be seen as the analogue of Lemma~\ref{L:B4} corresponding to the case~${j=N}$.
 
\begin{lemma}\label{L:B5} 
Let $n,\, m \in \N$, $3/2<s'<s<2$, and  $\nu\in(0,\infty)$ be given. 
Let further  $f\in H^s(\mathbb{R})$, $a\in H^{s-1}(\mathbb{R})$, and $\oo\in \{1\}\cup H^{s-1}(\mathbb{R})$.
For any sufficiently small $\e\in(0,1)$, there is
a constant $K$ depending on $\e,$ $ n,$ $ m,$ $ \|f\|_{H^s},$  $\|a\|_{H^{s-1}},$ and $ \|\oo\|_{H^{s-1}}$ (if $\oo\neq1$)  such that 
 \begin{equation*} 
\begin{aligned}
  &\|\pi_N^\e\oo B_{n,m}^0(f)[ ah]\|_{H^{s-1}}\leq \nu \|\pi_N^\e h\|_{H^{s-1}}+K\| h\|_{H^{s'-1}}
\end{aligned}  
 \end{equation*}
for  all $h\in H^{s-1}(\mathbb{R})$.
\end{lemma}  
\begin{proof}
It is suitable to decompose 
\begin{align*}
 \pi_j^\e\oo B_{n,m}^0(f)[ah]=\oo (T_1+ T_2), 
\end{align*}
where
\begin{align*}
T_1&:=\pi_j^\e B_{n,m}^0(f)[ ah]- B_{n,m}^0(f)[ \pi_j^\e ah],\\[1ex]
T_2&:= B_{n,m}^0(f)[\pi_j^\e a h].
\end{align*}
Lemma~\ref{L:B0} yields
\[
\|\oo T_1\|_{H^{s-1}}\leq K\|a h\|_2\leq K\|h\|_2.
\] 
Moreover, invoking Lemma~\ref{L:MP0}~(ii),  \eqref{MES}, the identity $\chi_j^\e\pi_j^\e=\pi_j^\e$, and the fact that $a$ vanishes at infinity,  we find
\begin{align*}
\|\oo T_2\|_{H^{s-1}}&\leq C\|\pi_j^\e a h\|_{H^{s-1}} \leq C\|\chi_j^\e a \|_\infty\|\pi_j^\e h\|_{H^{s-1}}+K\| h\|_{H^{s'-1}}\\[1ex]
&\leq  \nu  \|\pi_j^\e h\|_{H^{s-1}}+K\| h\|_{H^{s'-1}}
\end{align*}
 provided that $\e$ is sufficiently small.
This completes the proof.
\end{proof}

{\bf Acknowledgements:} The research leading to this paper was carried out while the second author enjoyed the hospitality of DFG Research Training Group 2339
``Interfaces, Complex Structures, and Singular Limits in Continuum Mechanics - Analysis and Numerics'' at the Faculty of Mathematics of Regensburg University. 

\bibliographystyle{abbrv}
\bibliography{MP}

\end{document}